%% file: neurips_2025.tex
\documentclass{article}

\usepackage[letterpaper,margin=1.0in]{geometry}
\usepackage[numbers,sort]{natbib}

 \usepackage{threeparttable}
 \usepackage{amssymb,amsthm}
 \usepackage{multirow}
 \usepackage{thm-restate}
 \usepackage{xcolor}
\usepackage{algpseudocode,algorithm}

\usepackage{graphicx}

\input{math}
\usepackage[utf8]{inputenc} 
\usepackage[T1]{fontenc}    
\usepackage{hyperref}       
\usepackage{url}            
\usepackage{booktabs}       
\usepackage{amsfonts}       
\usepackage{nicefrac}       
\usepackage{microtype}      
\usepackage{xcolor}         

\newif\ifacceptaiedits
\acceptaieditstrue
\ifacceptaiedits
  
  \newcommand{\aireplace}[2]{{\color{black}#1}}
\else
  
  \newcommand{\aireplace}[2]{{\color{blue}#1}}
\fi

\title{Stochastic Non-Smooth Non-Convex Optimization with Decision-Dependent Distributions}

\author{Anonymous Author(s)}

\author{
    Chengchang Liu \qquad Zongqi Wan \qquad Haishan Ye \qquad John C.S. Lui
 }
 \date{}
\begin{document}

\maketitle

\begin{abstract}
We study stochastic zeroth-order optimization with decision-dependent distributions, where the sampling law depends on the current decision and only noisy function values are available. 
For the non-smooth non-convex setting, we establish an explicit convergence guarantee for finding a $(\delta,\epsilon)$-Goldstein stationary point with stochastic zeroth-order oracle (SZO) complexity of $\gO(d^2\delta^{-3}\epsilon^{-3})$. 
In addition, we show that the above complexity can be achieved with single SZO feedback per iteration. 
We further extend the analysis to smooth and Hessian-Lipschitz objectives, obtaining complexities $\gO(d^2\epsilon^{-6})$ and $\gO(d^2\epsilon^{-9/2})$, respectively. In the Hessian-Lipschitz case, this improves the best-known dependence on $\epsilon$ for decision-dependent zeroth-order methods by a factor of $\epsilon^{-1/2}$.

\end{abstract}

\input{main_material}

\section{Experiments}
\label{sec:experiments}

We evaluate our proposed ZO-O2NC on two decision-dependent optimization problems: strategic classification and multi-product pricing. These benchmarks are also considered in \citet{hikima2025nonconvexdd, hikima2025guided}. We changed the loss function of the strategic classification experiment with hinge loss so that its objective becomes non-smooth to fit the setting we considered. The pricing experiment still uses the smooth objective as in \citep{hikima2025guided}.

{We compare both variants of {ZO-O2NC}, namely the two-point estimator in {ZO-O2NC (Option I)} and the one-point residual estimator in {ZO-O2NC (Option II)}, with algorithms including {ZO-CO}, {ZO-SPH}, {ZO-GA}~\citep{hikima2025zeroth}, and {ZO-OG}~\citep{liu2024twotimescale}. 
The methods in \citet{hikima2025zeroth,liu2024twotimescale} follow the gradient descent framework and are equipped with different zeroth-order gradient estimators, i.e. coordinate-wise two-point (CO), sphere two-point (SPH), Gaussian two-point (GA), and one-point zeroth-order (OG) estimators respectively.}
Our algorithm differs from them and follows the online-to-nonconvex conversion framework.
All experiments were run on a machine equipped with an Intel(R) Xeon(R) Platinum 8457C CPU with 48 processors and 95.7 GB of RAM.

\subsection{Strategic Classification}

\begin{figure*}[t]
\centering
\begin{tabular}{@{}cc@{}}
\includegraphics[width=0.46\textwidth]{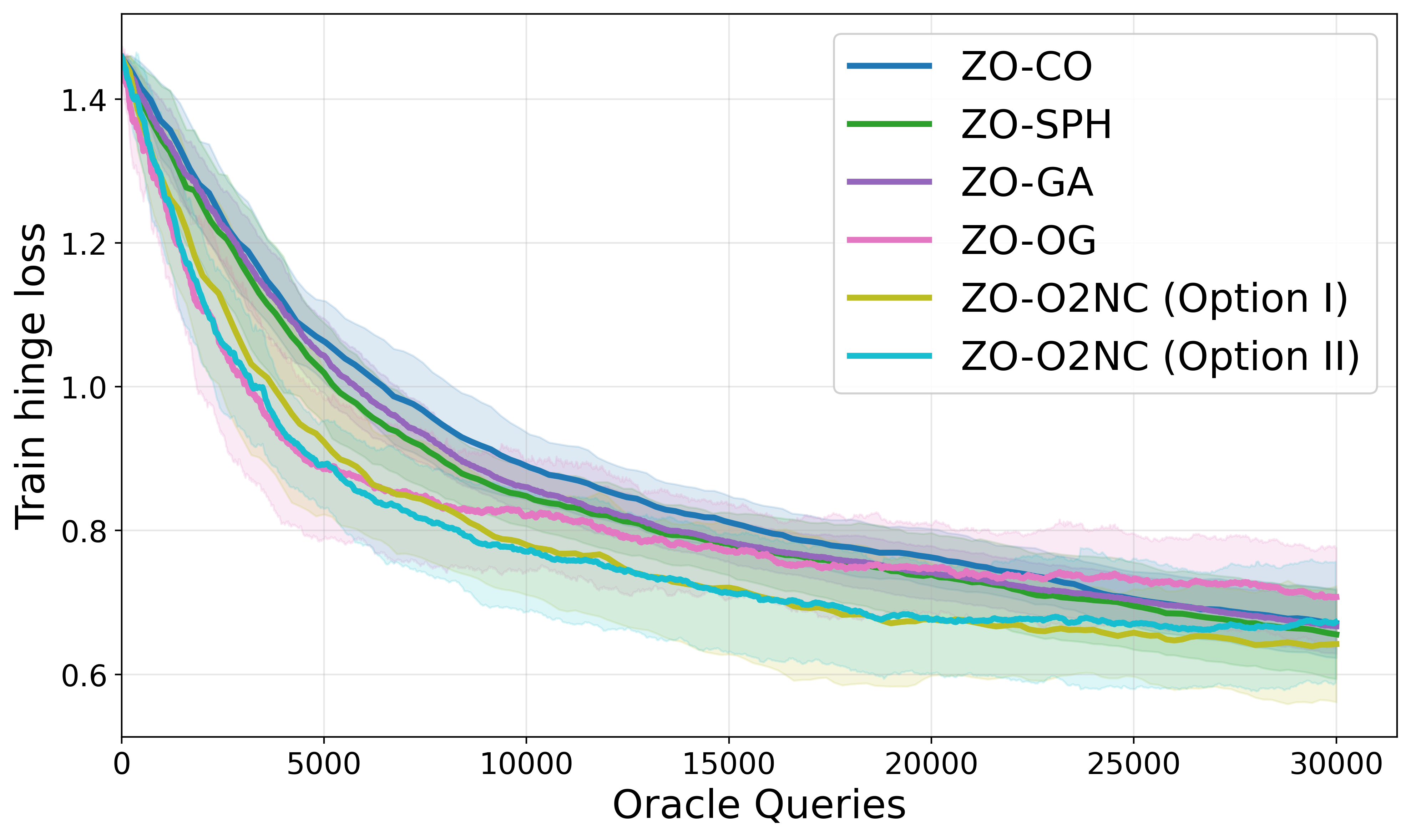} &
\includegraphics[width=0.46\textwidth]{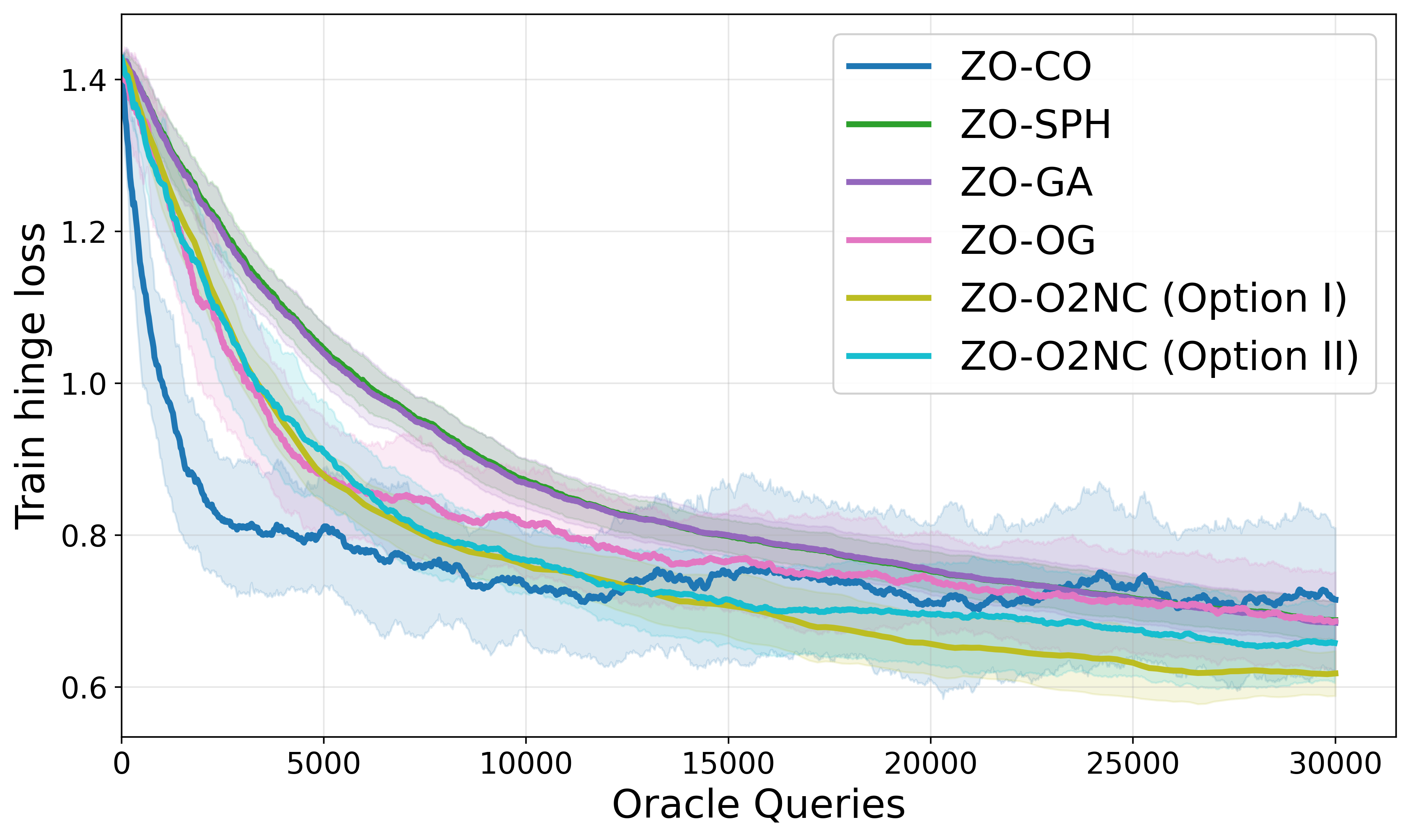} \\
\small (a) Random dataset with seed 101 & \small (b) Random dataset with seed 102 \\
\includegraphics[width=0.46\textwidth]{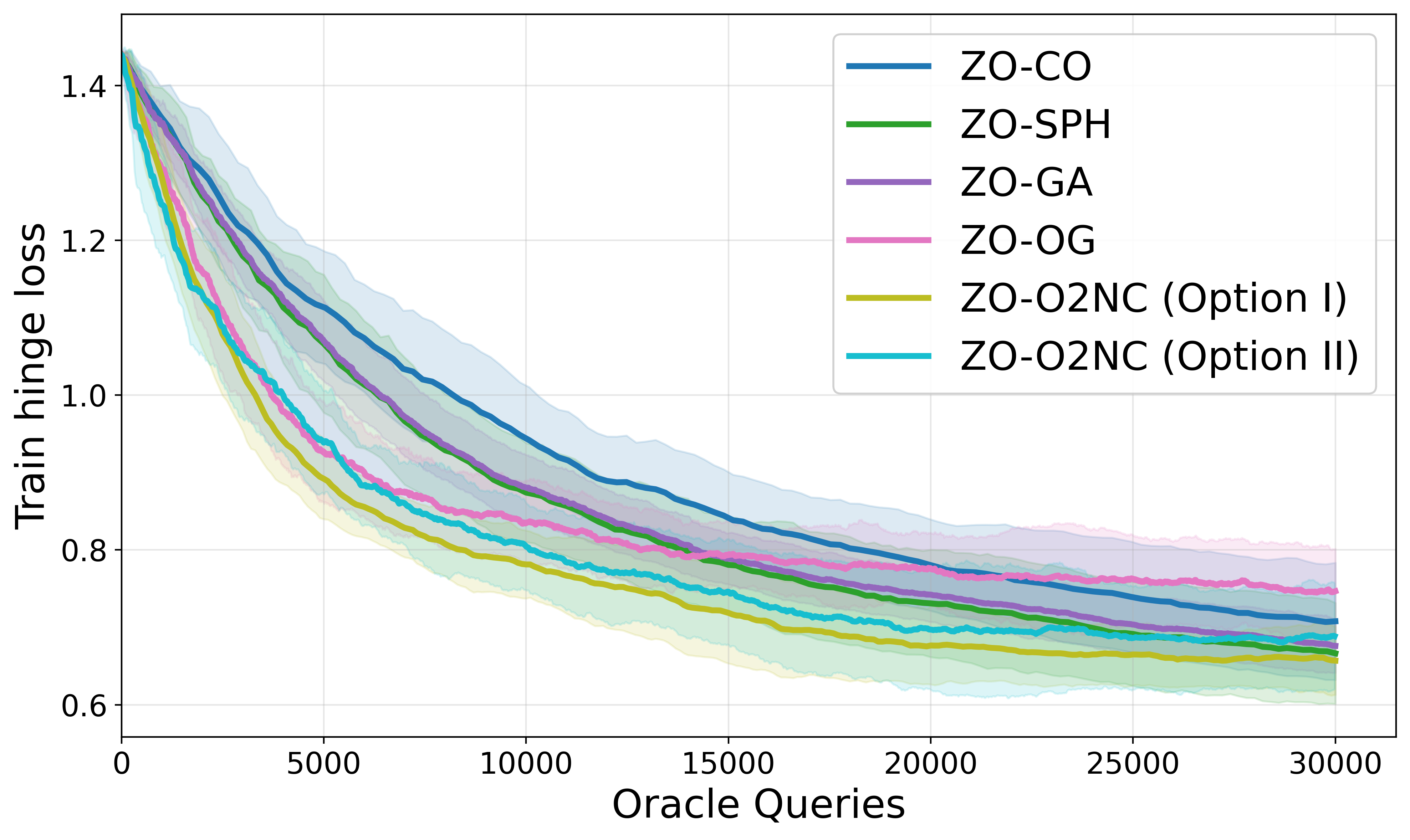} &
\includegraphics[width=0.46\textwidth]{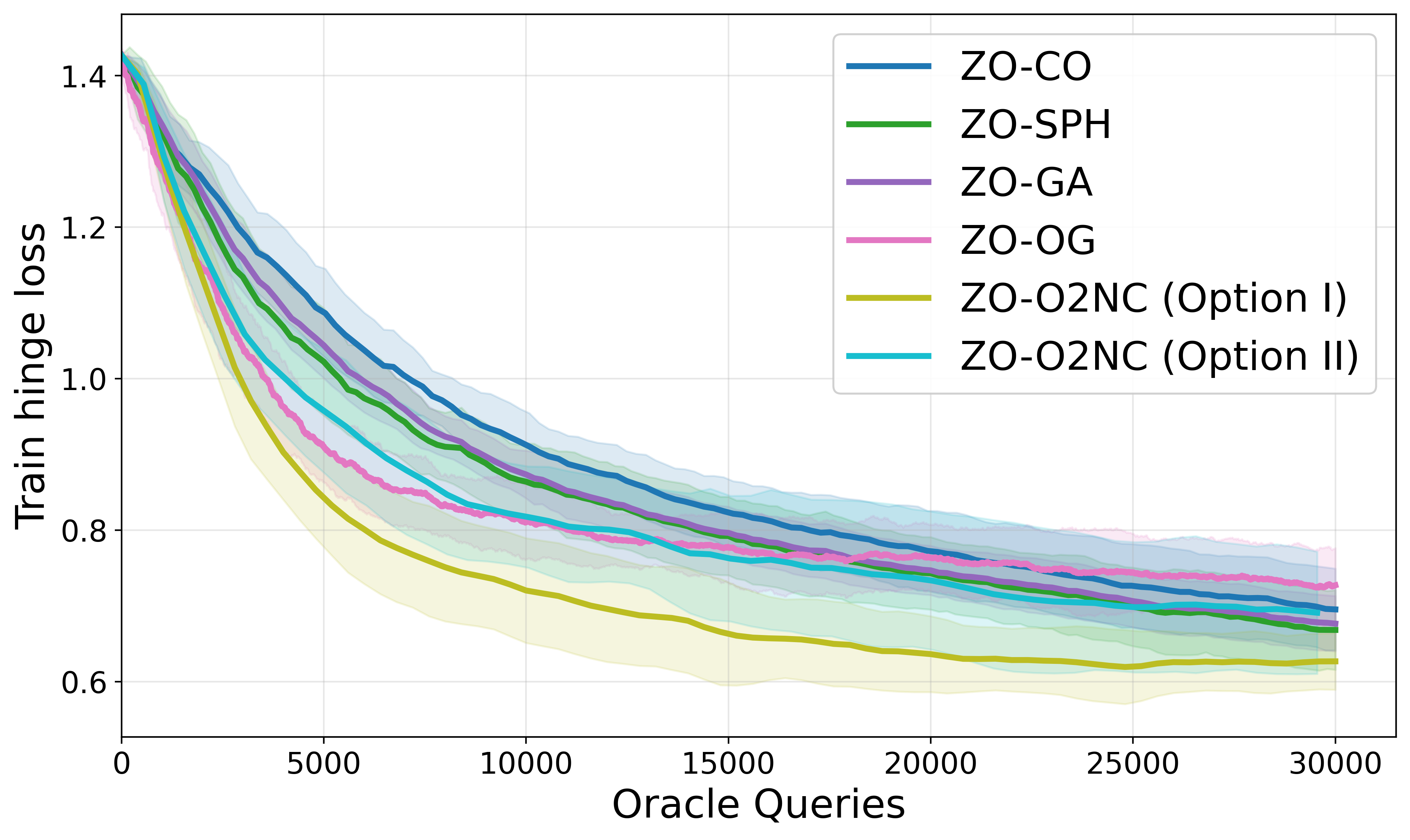} \\
\small (c) Random dataset with seed 103 & \small (d) Random dataset with seed 104
\end{tabular}
\caption{Strategic-classification train loss over oracle queries. }
\label{fig:strategic-train-loss}
\end{figure*}

Strategic classification models a binary decision system in which individuals may change their feature vectors after observing the classifier. For instance, when a bank decides whether to approve a loan, the applicant may falsify their credit history. In the experiment, the learner optimizes the classifier parameter $\vx=(\vx_{\rm feat},b)\in\sR^{12}$, while each data point has a mutable feature vector $\vw\in\sR^{11}$ and label $y\in\{-1,+1\}$.  Unlike ordinary supervised learning, the individual may modify their features after seeing the classifier in order to be classified to the favorable class $+1$. Thus the learner's objective is not the hinge risk under a fixed data distribution, but the hinge risk after the classifier-induced response map $\vw\mapsto \tilde{\vw}(\vx;\vw)$ has changed the observed features, which makes the training distribution depend on $\vx$.  We evaluate the post-response prediction with hinge loss, which makes the objective nonsmooth in addition to being decision-dependent; one stochastic oracle query samples one training example and returns its post-response hinge loss.
The training objective is
\[f(\vx)
    =
    \E_{(\vw_{\rm true},y)\sim D_{\rm train}}
    \left[
    \max\left\{0,1-y\big(\vx_{\rm feat}^{\top}\tilde{\vw}(\vx;\vw_{\rm true})+b\big)\right\}
    \right]
    \]
The full response rule and objective are given in Appendix~\ref{app:strategic-details}.

We use a hyperparameter grid search to tune the hyperparameters for each implemented algorithm, then report
their train loss and test accuracy in Figure~\ref{fig:strategic-train-loss} and Table~\ref{tab:strategic-by-split}.
The experiment dataset is from \citet{USL2019}, which contains the credit payment feature of $30000$ individuals. 
For each random seed in 101--104, we construct one fixed real-data instance by randomly sampling 13,272 records from the full dataset, then selecting a stratified test set of 1,000 records and using the remaining 12,272 records as the training set. The total query budget is fixed to $30000$, other details can be found in the Appendix~\ref{app:strategic-details}. 

The strategic-classification experiment demonstrates that two-point ZO-O2NC variant outperforms the baseline methods. 
ZO-O2NC with Option I obtains the lowest train loss on every dataset.  
The performance of single-point ZO-O2NC is also better than the priot single-point feedback method (ZO-OG).

\subsection{Multi-Product Pricing}

The pricing benchmark models a seller who chooses prices for multiple substitutable products while demand reacts to the chosen prices.  The decision variable is a price vector $\vx\in\sR^{30}$, one coordinate per product.  For each weekly instance, the reference price vector $\vtheta$ is built from real KSP-SP average prices \footnote{The price data are available at \url{https://www.ksp-sp.com/}.} and normalized into $[0.1,0.9]$; the remaining demand and cost parameters follow the same setup in \citet{hikima2025nonconvexdd}. Given $\vx$, each of $m=120$ buyers independently chooses either one product or the outside option according to a multinomial-logit (MNL) model with parameters detailed in the Appendix~\ref{app:pricing-details}.
Lowering a product's price relative to its reference price increases its purchase probability, while the outside option captures buyers who do not purchase any product.  A demand realization $\vxi=(\xi_1,\ldots,\xi_{30})$ is obtained by drawing 120 multinomial choices and counting how many buyers select each product.

For a realized demand vector $\vxi$, the seller's revenue is $\sum_i x_i\xi_i$ and the production cost is a separable piecewise-linear function $c(\vxi)=\sum_i c_i(\xi_i)$.  We minimize negative profit, so the stochastic oracle returns
\[
F(\vx;\vxi)=-\sum_{i=1}^{30}x_i\xi_i+c(\vxi),
\] 
with $\vxi\sim D(\vx)$ determined by the price-dependent MNL demand model. The objective of this experiment is $f(\vx) = \E_{\vxi\sim D(\vx)} [F(\vx;\vxi)]$.
The details of the problem setting can be found in Appendix~\ref{app:pricing-details}.

\begin{table*}[t]
\caption{Pricing final objective on real weekly instances.  
Each cell reports mean $\pm$ standard deviation over $10$ held-out seeds.}
\label{tab:pricing-objective}
\centering
\small
\resizebox{\textwidth}{!}{
\begin{tabular}{lcccccc}
\addlinespace
\toprule
Week & ZO-CO & ZO-SPH & ZO-GA & ZO-OG & ZO-O2NC (Option I) & ZO-O2NC (Option II) \\
\midrule
202610 & $-33.9138 \pm 0.3724$ & $-35.1343 \pm 0.1417$ & $-35.0369 \pm 0.1188$ & $-29.5875 \pm 1.4826$ & $\mathbf{-35.4249 \pm 0.0880}$ & $-35.1458 \pm 0.1819$ \\
202611 & $-22.2947 \pm 0.3391$ & $-22.3252 \pm 0.2732$ & $-22.3862 \pm 0.1218$ & $-21.9161 \pm 0.4503$ & $\mathbf{-22.8352 \pm 0.1006}$ & $-22.8308 \pm 0.1065$ \\
202612 & $-21.2159 \pm 0.3512$ & $-21.5353 \pm 0.1885$ & $-21.4584 \pm 0.1547$ & $-20.8643 \pm 0.2986$ & $\mathbf{-21.8371 \pm 0.1387}$ & $-21.6164 \pm 0.0867$ \\
202613 & $-18.3799 \pm 0.3636$ & $-18.5372 \pm 0.2821$ & $-18.5938 \pm 0.2408$ & $-17.8047 \pm 0.4510$ & $\mathbf{-18.8398 \pm 0.2717}$ & $-18.4587 \pm 0.3218$ \\
\bottomrule
\end{tabular}
}
\end{table*}

Table~\ref{tab:pricing-objective} reports the final objective on the four real weekly instances. Since the objective is negative profit, smaller values correspond to better pricing decisions. ZO-O2NC with Option I achieves the best final objective on all four weeks. ZO-O2NC with Option II is also competitive: it outperforms all baseline methods on three of the four weeks and remains competitive on the remaining week.

\section{Conclusions and Future Works}
\label{sec:conclu}
We studied stochastic zeroth-order optimization with decision-dependent distributions in the non-smooth non-convex setting. 
We established an explicit convergence guaranty for finding a $(\delta,\epsilon)$-Goldstein stationary point with an SZO complexity of $\gO(d^2\delta^{-3}\epsilon^{-3})$. 
In addition, we showed that such an SZO complexity can be achieved by single-point zeroth-order estimator.
We also extended the analysis to smooth and Hessian-Lipschitz objectives. 
In the smooth case, our framework achieves the state-of-the-art complexity $\gO(d^2\epsilon^{-6})$. 
In the Hessian-Lipschitz case, it improves the best known result from $\gO(d^2\epsilon^{-5})$ to $\gO(d^2\epsilon^{-9/2})$. 
These results show that the proposed ZO-O2NC method provides a unified guarantee to handle both non-smooth and smooth decision-dependent problems under zeroth-order feedback.

For the future work, it would be interesting to establish lower bounds for decision-dependent zeroth-order optimization, to understand whether the rates obtained here are optimal.
It is also possible to consider stochastic composite or constraint non-smooth non-convex problems with decision dependent distribution~\cite{liu2024gradient,liu2024zeroth}.

\bibliographystyle{plainnat}
\bibliography{sample}
\appendix

\input{Additional_related}

\input{appendix}

\section{Experiment Details}
\label{app:experiment-details}

This appendix gives the full experimental protocol used for the results in Section~\ref{sec:experiments}.  All methods are compared under the same stochastic-oracle query budget on each fixed problem instance.  The baselines are denoted by ZO-CO, ZO-SPH, ZO-GA, and ZO-OG, corresponding respectively to the coordinate-wise two-point, sphere two-point, Gaussian two-point, and one-point zeroth-order estimators.  We report ZO-O2NC (Option I) and ZO-O2NC (Option II) as separate methods.  Hyperparameters are tuned separately for each base instance using development randomness, and the reported numbers are computed only on held-out randomness.
\subsection{Strategic Classification}
\label{app:strategic-details}

\paragraph{Problem setting.}
The strategic-classification experiment is built from the processed credit dataset used in \citet{USL2019}.  We use four fixed real-data instances indexed by seeds $101,102,103,104$.  For each seed, we sample 13,272 records, reserve a stratified test set of 1,000 records, and use the remaining 12,272 records as the training set.  The latent feature vector $\vw$ contains 11 mutable credit features. These features are standardized using the training-set mean and standard deviation for the corresponding split.  The decision variable is a linear classifier $\vx=(\vx_{\rm feat},b)\in\sR^{12}$, where $\vx_{\rm feat}\in\sR^{11}$ and $b\in\sR$.

\begin{table}[t]
\caption{Strategic-classification results on each real credit split.  Each cell reports the mean $\pm$ standard deviation.  Lower is better for train loss and test loss; higher is better for test accuracy and test AUC. The best method in each row is bolded.}
\label{tab:strategic-by-split}
\centering
\scriptsize
\setlength{\tabcolsep}{3pt}
\textbf{Split 101}
\resizebox{\textwidth}{!}{
\begin{tabular}{lcccccc}
\toprule
Metric & ZO-CO & ZO-SPH & ZO-GA & ZO-OG & \aireplace{ZO-O2NC (Option I)}{ZO-O2NC (ours)} & \aireplace{ZO-O2NC (Option II)}{} \\
\midrule
Train loss & $0.6711 \pm 0.0482$ & $0.6555 \pm 0.0618$ & $0.6669 \pm 0.0379$ & $0.7074 \pm 0.0687$ & $\mathbf{0.6422 \pm 0.0807}$ & \aireplace{$0.6727 \pm 0.0843$}{} \\
Test accuracy & $\mathbf{0.7705 \pm 0.0128}$ & $0.7665 \pm 0.0196$ & $0.7643 \pm 0.0231$ & $0.7552 \pm 0.0336$ & $0.7634 \pm 0.0205$ & \aireplace{$0.7675 \pm 0.0260$}{} \\
Test loss & $0.6717 \pm 0.0482$ & $0.6615 \pm 0.0596$ & $0.6683 \pm 0.0348$ & $0.7172 \pm 0.0747$ & $\mathbf{0.6448 \pm 0.0758}$ & \aireplace{$0.6824 \pm 0.0823$}{} \\
Test AUC & $0.6725 \pm 0.0248$ & $0.6724 \pm 0.0198$ & $0.6761 \pm 0.0163$ & \aireplace{$0.6763 \pm 0.0439$}{$\mathbf{0.6763 \pm 0.0439}$} & $0.6761 \pm 0.0276$ & \aireplace{$\mathbf{0.6860 \pm 0.0270}$}{} \\
\bottomrule
\end{tabular}
}

\vspace{0.6em}
\textbf{Split 102}
\resizebox{\textwidth}{!}{
\begin{tabular}{lcccccc}
\toprule
Metric & ZO-CO & ZO-SPH & ZO-GA & ZO-OG & \aireplace{ZO-O2NC (Option I)}{ZO-O2NC (ours)} & \aireplace{ZO-O2NC (Option II)}{} \\
\midrule
Train loss & $0.7152 \pm 0.0946$ & $0.6881 \pm 0.0245$ & $0.6848 \pm 0.0274$ & $0.6869 \pm 0.0643$ & $\mathbf{0.6181 \pm 0.0296}$ & \aireplace{$0.6584 \pm 0.0527$}{} \\
Test accuracy & $0.7638 \pm 0.0247$ & \aireplace{$0.7625 \pm 0.0151$}{$0.7624 \pm 0.0151$} & $0.7613 \pm 0.0203$ & $0.7747 \pm 0.0204$ & $\mathbf{0.7838 \pm 0.0077}$ & \aireplace{$0.7756 \pm 0.0122$}{} \\
Test loss & $0.6838 \pm 0.0994$ & $0.6820 \pm 0.0227$ & $0.6795 \pm 0.0245$ & $0.6697 \pm 0.0692$ & $\mathbf{0.6115 \pm 0.0298}$ & \aireplace{$0.6496 \pm 0.0535$}{} \\
Test AUC & $0.6941 \pm 0.0479$ & $0.6796 \pm 0.0172$ & $0.6809 \pm 0.0158$ & $0.7090 \pm 0.0252$ & $\mathbf{0.7138 \pm 0.0233}$ & \aireplace{$0.7001 \pm 0.0302$}{} \\
\bottomrule
\end{tabular}
}

\vspace{0.6em}
\textbf{Split 103}
\resizebox{\textwidth}{!}{
\begin{tabular}{lcccccc}
\toprule
Metric & ZO-CO & ZO-SPH & ZO-GA & ZO-OG & \aireplace{ZO-O2NC (Option I)}{ZO-O2NC (ours)} & \aireplace{ZO-O2NC (Option II)}{} \\
\midrule
Train loss & $0.7078 \pm 0.0757$ & $0.6664 \pm 0.0652$ & $0.6761 \pm 0.0346$ & $0.7467 \pm 0.0547$ & $\mathbf{0.6570 \pm 0.0440}$ & \aireplace{$0.6886 \pm 0.0685$}{} \\
Test accuracy & $0.7585 \pm 0.0190$ & $0.7663 \pm 0.0201$ & $0.7691 \pm 0.0218$ & $0.7708 \pm 0.0258$ & \aireplace{$\mathbf{0.7776 \pm 0.0127}$}{$\mathbf{0.7777 \pm 0.0127}$} & \aireplace{$0.7707 \pm 0.0136$}{} \\
Test loss & $0.7306 \pm 0.0873$ & $0.6803 \pm 0.0821$ & $0.6892 \pm 0.0462$ & $0.7470 \pm 0.0560$ & $\mathbf{0.6453 \pm 0.0488}$ & \aireplace{$0.6829 \pm 0.0818$}{} \\
Test AUC & $0.6499 \pm 0.0341$ & $0.6660 \pm 0.0492$ & $0.6691 \pm 0.0308$ & $0.6875 \pm 0.0286$ & $\mathbf{0.7083 \pm 0.0208}$ & \aireplace{$0.6873 \pm 0.0443$}{} \\
\bottomrule
\end{tabular}
}

\vspace{0.6em}
\textbf{Split 104}
\resizebox{\textwidth}{!}{
\begin{tabular}{lcccccc}
\toprule
Metric & ZO-CO & ZO-SPH & ZO-GA & ZO-OG & \aireplace{ZO-O2NC (Option I)}{ZO-O2NC (ours)} & \aireplace{ZO-O2NC (Option II)}{} \\
\midrule
Train loss & $0.6952 \pm 0.0539$ & $0.6681 \pm 0.0526$ & $0.6764 \pm 0.0367$ & $0.7272 \pm 0.0490$ & $\mathbf{0.6266 \pm 0.0375}$ & \aireplace{$0.6909 \pm 0.0806$}{} \\
Test accuracy & $0.7617 \pm 0.0208$ & $0.7719 \pm 0.0181$ & $0.7749 \pm 0.0107$ & $0.7745 \pm 0.0163$ & $\mathbf{0.7823 \pm 0.0104}$ & \aireplace{$0.7717 \pm 0.0202$}{} \\
Test loss & $0.7025 \pm 0.0516$ & $0.6738 \pm 0.0509$ & $0.6836 \pm 0.0345$ & $0.7373 \pm 0.0548$ & $\mathbf{0.6374 \pm 0.0456}$ & \aireplace{$0.6997 \pm 0.0810$}{} \\
Test AUC & $0.6592 \pm 0.0238$ & $0.6731 \pm 0.0273$ & $0.6762 \pm 0.0165$ & $0.6878 \pm 0.0233$ & $\mathbf{0.6991 \pm 0.0246}$ & \aireplace{$0.6778 \pm 0.0416$}{} \\
\bottomrule
\end{tabular}
}
\end{table}

Each example consists of a latent feature vector $\vw_{\rm true}\in\sR^{11}$ and a label $y\in\{-1,+1\}$.  Given $\vx$, the individual may strategically alter its observable features according to the best-response problem
\[
    \tilde{\vw}(\vx;\vw_{\rm true})
    \in \arg\max_{\vw\in\sR^{11}}
    \left\{ r(\vw;\vx)-\|\vw-\vw_{\rm true}\|_2^2 \right\},
\]
where the approval reward is
\[
    r(\vw;\vx)=
    \begin{cases}
    \tau, & \vx_{\rm feat}^{\top}\vw+b\ge 0,\\
    0, & \vx_{\rm feat}^{\top}\vw+b<0,
    \end{cases}
    \qquad \tau=2.
\]
Equivalently, an already approved agent keeps its original features.  An unapproved agent moves to the Euclidean projection onto the decision boundary only if the squared movement cost is at most $\tau$; otherwise it does not move.  This response rule makes the data distribution depend on $\vx$.
In closed form, if $s=\vx_{\rm feat}^{\top}\vw_{\rm true}+b<0$ and $\|\vx_{\rm feat}\|_2>0$, the closest point on the decision boundary is
\[
    \vw_{\rm proj}
    =
    \vw_{\rm true}
    -
    \frac{s}{\|\vx_{\rm feat}\|_2^2}\vx_{\rm feat}.
\]
The individual chooses $\tilde{\vw}=\vw_{\rm proj}$ when $\|\vw_{\rm proj}-\vw_{\rm true}\|_2^2\le \tau$ and otherwise chooses $\tilde{\vw}=\vw_{\rm true}$.  If $s\ge0$ or $\|\vx_{\rm feat}\|_2=0$, then $\tilde{\vw}=\vw_{\rm true}$.

Our stochastic objective is the training hinge loss after strategic response:
\[
    f(\vx)
    =
    \E_{(\vw_{\rm true},y)\sim D_{\rm train}}
    \left[
    \max\left\{0,1-y\big(\vx_{\rm feat}^{\top}\tilde{\vw}(\vx;\vw_{\rm true})+b\big)\right\}
    \right].
\]
One stochastic oracle query samples one training index uniformly and returns the corresponding hinge loss at the current classifier parameter.  We report final train loss, test accuracy, test loss, and test AUC.  This experiment differs from \citet{hikima2025nonconvexdd} only in the loss: their experiment uses a smooth classification loss, while our experiment uses hinge loss to create a nonsmooth decision-dependent objective.

\paragraph{Hyperparameter grids.}
For each credit split, all methods start from $\vone\in\sR^{12}$ and use a query budget of $Q=30{,}000$.  We use 5 development runs to choose hyperparameters with a grid search method and 20 held-out runs to report the final metrics.  For ZO-SPH and ZO-GA, $N$ denotes the number of random directions and the mini-batch parameter is fixed to $m=1$; for ZO-CO and ZO-OG, $m$ denotes the mini-batch parameter and $N$ is fixed to 1. 
The hyperparameter grid is $\eta\in\{0.01,0.1,1.0\}$, $\mu\in\{0.1,0.5,1.0,2.0,4.0\}$, with $N\in\{1,10,100\}$ for ZO-SPH/ZO-GA and $m\in\{1,10,100\}$ for ZO-CO/ZO-OG. 
For both ZO-O2NC variants, we tune $\delta\in\{0.5,1.0,2.0,5.0\}$, $m\in\{1,5,10,20,50\}$, $T\in\{5,10,20\}$, and $\eta_{\rm online}\in\{10^{-4},10^{-3},10^{-2}\}$; the inner radius is $D=\delta/T$. The selected configuration is the one with the smallest mean final training hinge loss on the development runs.

\paragraph{Detailed results.}
Table~\ref{tab:strategic-by-split} reports the per-split strategic-classification results for seeds 101--104, with ZO-O2NC (Option I) and ZO-O2NC (Option II) shown as separate methods.  For each split and method, we report final train hinge loss, test accuracy, test hinge loss, and test AUC as the mean $\pm$ standard deviation over 20 held-out runs.

\subsection{Multi-product Pricing}
\label{app:pricing-details}

\paragraph{Problem setting.}
The pricing experiment uses the multi-product pricing model of \citet{hikima2025nonconvexdd}.  There are $n=30$ substitutable products and $m_{\rm buyer}=120$ potential buyers.  The decision variable is the price vector $\vx=(x_1,\ldots,x_n)\in\sR^n$, so the optimization dimension is $d=30$.  We use four fixed weekly instances, corresponding to KSP-SP weeks 202610, 202611, 202612, and 202613.  For each weekly instance, we take the top 30 recorded average confectionery prices and min-max normalize them into $[0.1,0.9]$ to obtain the reference price vector $\vtheta=(\theta_1,\ldots,\theta_n)$.

Given a queried price vector $\vx$, each buyer independently either purchases one product or chooses the outside option.  Conditional on $\vx$, the product probabilities follow a multinomial-logit model:
\[
    p_i(\vx)=
    \frac{\exp\{\gamma_i(\theta_i-x_i)\}}
    {a_0+\sum_{j=1}^n\exp\{\gamma_j(\theta_j-x_j)\}},
    \qquad i=1,\ldots,n,
\]
and the no-purchase probability is
\[
    p_0(\vx)=
    \frac{a_0}
    {a_0+\sum_{j=1}^n\exp\{\gamma_j(\theta_j-x_j)\}}.
\]
We set $a_0=0.1n=3$ and $\gamma_i=2\pi/(\sqrt{6}\theta_i)$.  The demand vector $\vxi=(\xi_1,\ldots,\xi_n)$ is obtained by drawing $m_{\rm buyer}$ multinomial choices with probabilities $(p_1(\vx),\ldots,p_n(\vx),p_0(\vx))$ and retaining the product counts.  Thus $\sum_i \xi_i\le m_{\rm buyer}$.

For a realized demand vector $\vxi$, the seller's revenue is
\[
    s(\vx,\vxi)=\sum_{i=1}^n x_i\xi_i,
\]
and the production cost is separable, $c(\vxi)=\sum_{i=1}^n c_i(\xi_i)$.  As in \citet{hikima2025nonconvexdd}, each product has thresholds $l_i=0.5m_{\rm buyer}/n=2$ and $u_i=1.5m_{\rm buyer}/n=6$.  The per-product cost is piecewise linear, with a larger marginal cost in the high-demand regime:
\[
c_i(z)=
\begin{cases}
2w_i z, & z\le l_i,\\
w_i(z-l_i)+2w_i l_i, & l_i<z\le u_i,\\
3w_i(z-u_i)+w_i(u_i-l_i)+2w_i l_i, & z>u_i,
\end{cases}
\]
where $w_i=\rho_i\theta_i$ and $\rho_i\sim{\rm Unif}[0.25,0.5]$ is drawn once and fixed within the instance.  The induced demand distribution depends on $\vx$, and the stochastic oracle returns the negative profit
\[
    F(\vx;\vxi)=-s(\vx,\vxi)+c(\vxi),
\]
where $\vxi\sim D(\vx)$ is the multinomial demand vector generated by the MNL model.  The objective is
\[
    f(\vx)=\E_{\vxi\sim D(\vx)}[F(\vx;\vxi)].
\]
For final evaluation and plots, we estimate $f(\vx)$ with an independent 1000-sample Monte Carlo average at the reported point.  These evaluation samples are not counted as optimization oracle queries.

\paragraph{Hyperparameter grids.} For each weekly pricing instance, all methods start from $0.5\vone\in\sR^{30}$ and use a query budget of $Q=5{,}000$.  We use 3 development runs to choose hyperparameters and 10 held-out runs to report the final objective.  As in the strategic-classification benchmark, $N$ is tuned for ZO-SPH/ZO-GA with $m=1$, while $m$ is tuned for ZO-CO/ZO-OG with $N=1$.  The hyperparameter grid for algorithms in \citep{hikima2025nonconvexdd} is $\eta\in\{10^{-5},10^{-4},10^{-3},10^{-2}\}$ and $\mu\in\{0.004,0.02,0.1,0.5,2.5\}$, with $N\in\{1,10,100\}$ for ZO-SPH/ZO-GA and $m\in\{1,10,100\}$ for ZO-CO/ZO-OG. 
For both ZO-O2NC variants, we tune $\delta\in\{0.01,0.05,0.1,0.5,1.0\}$, $m\in\{1,5,10,20\}$, $T\in\{10,20,50\}$, and $\eta_{\rm online}\in\{10^{-4},10^{-3},10^{-2}\}$. The selected configuration is the one with the smallest mean final objective on the development runs.

\end{document}

%% file: math.tex
\usepackage{amsmath,amsfonts,bm,amssymb}

\def\eqref#1{(\ref{#1})}

\def\1{\bm{1}}

\def\vzero{{\bm{0}}}
\def\vone{{\bm{1}}}

\def\vtheta{{\bm{\theta}}}
\def\vxi{{\bm{\xi}}}
\def\vDelta{{\bm{\Delta}}}

\def\vg{{\bm{g}}}

\def\vq{{\bm{q}}}
\def\vr{{\bm{r}}}

\def\vu{{\bm{u}}}
\def\vv{{\bm{v}}}
\def\vw{{\bm{w}}}
\def\vx{{\bm{x}}}
\def\vy{{\bm{y}}}
\def\vz{{\bm{z}}}

\DeclareMathAlphabet{\mathsfit}{\encodingdefault}{\sfdefault}{m}{sl}
\SetMathAlphabet{\mathsfit}{bold}{\encodingdefault}{\sfdefault}{bx}{n}

\def\gO{{\mathcal{O}}}

\def\sB{{\mathbb{B}}}

\def\sN{{\mathbb{N}}}

\def\sR{{\mathbb{R}}}
\def\sS{{\mathbb{S}}}

\newcommand{\E}{\mathbb{E}}

\newcommand{\Var}{\mathrm{Var}}

\newtheorem{thm}{Theorem}[section]
\newtheorem{dfn}{Definition}[section]

\newtheorem{lem}{Lemma}[section]
\newtheorem{asm}{Assumption}[section]
\newtheorem{rmk}{Remark}[section]

\newtheorem{prop}{Proposition}[section]

\newcommand{\norm}[1]{\left\| #1 \right\|}

\def\1{\bf{1}}

\def\inner#1#2{\langle #1, #2 \rangle}

\def\RB{{\mathbb R}}

%% file: main_material.tex
\section{Introduction}

This paper studies the stochastic optimization problem
\begin{align}
    \label{prob:main}
    \min_{\vx \in \sR^d} f(\vx) \triangleq \E_{\xi \sim \Xi(\vx)}[F(\vx;\xi)],
\end{align}
where the sampling distribution $\Xi(\vx)$ depends on the decision $\vx$, and only stochastic function values are available. We allow the expected objective $f$ to be non-smooth and non-convex.
Decision-dependent models arise naturally when deployed decisions change the future data-generating process~\cite{berahas2022theoretical,perdomo2020performative,mendler2020stochastic,chen2023performative, drusvyatskiy2023stochastic}, as in price optimization~\cite{ray2022decision} and strategic classification \citep{hikima2025zeroth,levanon2021strategic}. 
In these settings, each stochastic sample is observed only after querying the system at the current decision.

The main difficulty in solving \eqref{prob:main} is the \emph{distribution drift} caused by the dependence of $\Xi(\vx)$. If $\Xi(\vx)$ admits a density $p_{\vx}$ with respect to a reference measure, then
\begin{align*}
    \nabla f(\vx) = \E_{\xi \sim \Xi(\vx)}[\nabla_{\vx} F(\vx;\xi)] + \int F(\vx;\xi)\nabla_{\vx} p_{\vx}(\xi)\,d\xi.
\end{align*}
The second term depends on the derivative of the unknown sampling law and cannot be recovered from stochastic first-order information alone.
Therefore, it is difficult, and in general even impossible, to construct a reliable estimator of $\nabla f(\vx)$ from the stochastic first-order oracle $\nabla_{\vx} F(\vx;\xi)$.
This leads us to consider stochastic zeroth-order methods, where we assume access to the following stochastic zeroth-order oracle (SZO).
\begin{asm}\label{asm:estimator}
We assume the stochastic zeroth-order oracle $F(\vx;\xi)$ satisfies
\begin{align*}
    \E_{\xi \sim \Xi(\vx)}[F(\vx;\xi)] = f(\vx)~~\text{and}~~
    \E_{\xi \sim \Xi(\vx)} \big[ |F(\vx;\xi)-f(\vx)|^2 \big] \le \sigma^2~~~\text{for all}~~\vx\in\RB^d.
\end{align*}
\end{asm}

Several zeroth-order methods have been proposed to find an $\epsilon$-stationary point of \eqref{prob:main}.
\citet{liu2024twotimescale} establish a rate of $\widetilde{\gO}(\ell^6 d^2\epsilon^{-6})$, where $\ell \triangleq \sup_{\vx,\xi}|F(\vx;\xi)|$, which can be very large.
\citet{hikima2025nonconvexdd} remove the bounded-value assumption on $|F(\vx;\xi)|$ and prove an $\gO(d^{9/2}\epsilon^{-6})$ SZO complexity.
This rate is improved to $\gO(d^{4}\epsilon^{-6})$ by \citet{hikima2025guided}, and \citet{hikima2025zeroth} further obtain $\gO(d^2\epsilon^{-6})$.
\citet{hikima2025zeroth} also show that the SZO complexity can be reduced to $\gO(d^2\epsilon^{-5})$ under Hessian Lipschitzness. 
However, all of the above decision-dependent results assume that the objective function is smooth.
It remains unclear whether one can provide an explicit convergence guarantee for stochastic non-smooth non-convex problems.
This motivates us to ask the following question.

\textit{
Can one design efficient zeroth-order methods for stochastic non-smooth non-convex optimization with decision-dependent distributions?
}

For non-smooth non-convex problems, it is natural to seek a $(\delta,\epsilon)$-Goldstein stationary point, following \citet{goldstein1977optimization} and \citet{zhang2020complexity}.
In the decision-independent setting, where the sampling distribution is fixed (i.e., $\Xi(\vx)\equiv\Xi$), many zeroth-order methods have been developed for finding an approximate Goldstein stationary point.
\citet{lin2022gradient} transform the problem into finding an $\epsilon$-stationary point of the surrogate function $f_{\delta}$ (see Definition~\ref{dfn:f-delta}) and obtain an $\mathcal{O}(d^{3/2}\delta^{-1}\epsilon^{-4})$ SZO complexity.
This complexity is further improved to $\mathcal{O}(d^{3/2}\delta^{-1}\epsilon^{-3})$ by variance reduction~\cite{chen2023faster} and to $\mathcal{O}(d\delta^{-1}\epsilon^{-3})$ by the online-to-non-convex conversion (O2NC) framework~\cite{kornowski2024algorithm,cutkosky2023optimal}.
However, all these results assume that $F(\vx;\xi)$ is Lipschitz continuous, which does not generally hold in the decision-dependent setting considered here.

Moreover, the above methods for decision-independent non-smooth non-convex optimization require at least two-point feedback per iteration to obtain a good approximation of $\nabla f_{\delta}(\vx)$.
This can be restrictive in settings where only one function evaluation is available at each iteration~\cite{flaxman2005online,shamir2013complexity}.
In addition, existing one-point methods for smooth non-convex optimization with decision-dependent distributions often have strong dependence on either the magnitude of the SZO values~\cite{hikima2025nonconvexdd} or the dimension~\cite{liu2024twotimescale}, due to the large variance of one-point gradient estimators~\cite{larson2019derivative}.
This motivates us to further consider the following question.

\textit{
Can one-point-feedback zeroth-order methods achieve comparable SZO complexity for (non)-smooth non-convex problems?
}

In this paper, we answer the above two questions by studying how the two-point and one-point estimators approximate $\nabla f_{\delta}(\vx)$, and by showing how they can be combined with the O2NC framework to solve non-smooth non-convex problems with decision-dependent distributions.
We summarize our contributions below and compare our methods with prior work in Table~\ref{tbl:compare}.

\begin{table}[t]
\centering
\caption{Comparison of zeroth-order methods for stochastic decision-dependent non-convex problems. ``NS-NC'', ``S-NC'', and ``HL-NC'' denote the non-smooth non-convex, smooth non-convex, and Hessian-Lipschitz non-convex settings, respectively.
For the NS-NC setting, we report the SZO complexity for finding a $(\delta,\epsilon)$-Goldstein stationary point. For the other settings, we report the SZO complexity for finding an $\epsilon$-stationary point.}
\label{tbl:compare}
\begin{threeparttable}
\centering
\begin{tabular}{ccccc}
\toprule
Method & NS-NC & S-NC & HL-NC & SZO Feedback/Iter. \\
\midrule
\midrule
\citet{liu2024twotimescale} & -- & $\widetilde{\gO}(d^2\epsilon^{-6})$$^\dag$ & -- & $1$ \\\addlinespace
 \citet{hikima2025nonconvexdd} & -- & $\gO(d^{9/2}\epsilon^{-6})$ & -- & $1$ \\\addlinespace
\citet{hikima2025guided} & -- & $\gO(d^{4}\epsilon^{-6})$ & -- & $2$ \\\addlinespace
\citet{hikima2025zeroth} & -- & $\gO(d^3\epsilon^{-6})$ & $\gO(d^3\epsilon^{-5})$ & $2d$ \\\addlinespace
\citet{hikima2025zeroth} & $\gO(d^{5/2}\delta^{-3}\epsilon^{-4})$$^*$ & $\gO(d^2\epsilon^{-6})$ & $\gO(d^2\epsilon^{-5})$ & $2N$$^\ddag$ \\\addlinespace
\hline
\addlinespace
Algorithm~\ref{alg:O2NC-detail} (Option~I) & $\gO(d^2\delta^{-3}\epsilon^{-3})$ & $\gO(d^2\epsilon^{-6})$ & $\gO(d^2\epsilon^{-9/2})$ & $2$ \\\addlinespace
Algorithm~\ref{alg:O2NC-detail} (Option~II) & $\gO(d^2\delta^{-3}\epsilon^{-3})$ & $\gO(d^2\epsilon^{-6})$ & $\gO(d^2\epsilon^{-9/2})$ & $1$ \\
\bottomrule
\end{tabular}
\begin{tablenotes}[flushleft]
\footnotesize
\item[$^\dag$] This result requires additional assumption that $|F(\vx;\xi)|\leq l$ and the complexity omits the $l^6$ dependency.
\item[$^*$] \citet{hikima2025zeroth} did \textbf{not} provide a convergence guarantee for the NS-NC setting. We provide the convergence guarantee in Appendix~\ref{app:sgd-baseline}.
\item[$^\ddag$] This requires to set $N = \mathcal{O}(d^2\epsilon^{-4})$.
\end{tablenotes}
\end{threeparttable}
\end{table}

\begin{itemize}
    \item In Section~\ref{sec:zo-estimators}, we study how the two-point estimator and the one-point residual estimator approximate the gradient of the smoothed surrogate under decision-dependent sampling.
    \item In Section~\ref{sec:zo-convergence}, we combine these gradient estimators with the O2NC framework and provide the first explicit convergence guarantee for stochastic non-smooth non-convex optimization with decision-dependent distributions, with SZO complexity of $\gO(d^2\sigma^2\delta^{-3}\epsilon^{-3})$ for finding a $(\delta,\epsilon)$-Goldstein stationary point.
    \item In Section~\ref{sec:smooth-extension}, we develop a general reduction framework for the smooth and Hessian-Lipschitz settings. We show that the single-feedback zeroth-order method matches state-of-the-art SZO complexity of the prior methods and improves the highly smooth rate from $\gO(d^2\epsilon^{-5})$ to $\gO(d^2\sigma^2\epsilon^{-9/2})$.
\end{itemize}


\section{Preliminaries}
\label{sec:preliminaries}
Throughout, $\norm{\cdot}$ denotes the Euclidean norm. For $\vx \in \sR^d$ and $r>0$, define
\begin{align*}
    \sB(\vx,r) \triangleq \{ \vy \in \sR^d : \norm{\vy-\vx} \le r \}, \qquad
    \sB^d \triangleq \sB(\vzero,1), \qquad
    \sS^{d-1} \triangleq \{ \vu \in \sR^d : \norm{\vu}=1 \}.
\end{align*}
For a nonempty set $S \subseteq \sR^d$, let
$
    {\rm dist}(\vx,S) \triangleq \inf_{\vy \in S} \norm{\vx-\vy}.
$
We also let $\Pi_S$ denote the Euclidean projection onto a closed convex set $S$.
All expectations are denoted by $\E[\cdot]$, with subscripts added when the underlying randomness needs to be emphasized. For a locally Lipschitz function $h$, $\partial h(\vx)$ denotes its Clarke subdifferential.

We make the following assumptions on the objective function \eqref{prob:main}.
\begin{asm} \label{asm:lower}
The objective function is lower bounded such that $
    f^* \triangleq \inf_{\vx \in \sR^d} f(\vx) > -\infty$.
\end{asm}

\begin{asm}\label{asm:lip_function}
The objective function is $L$-Lipschitz continuous on $\sR^d$, i.e., for all $\vx,\vy \in\sR^d$, it holds that
\begin{align*}
    |f(\vx)-f(\vy)| \le L \norm{\vx-\vy}.
\end{align*}
\end{asm}

We then introduce the definition of Clarke and Goldstein subdifferentials~\cite{clarke1990optimization,goldstein1977optimization}.
\begin{dfn}[Clarke and Goldstein subdifferentials]
For a locally Lipschitz function $f$ and $\vx \in \sR^d$, the Clarke subdifferential is
\begin{align*}
    \partial f(\vx)
    \triangleq
    \operatorname{conv}
    \left\{
        \vg :
        \exists \vx_n \to \vx,\;
        \nabla f(\vx_n)\; \text{exists},\;
        \nabla f(\vx_n) \to \vg
    \right\}.
\end{align*}
For $\delta>0$, the Goldstein $\delta$-subdifferential is
\begin{align*}
    \partial_\delta f(\vx)
    \triangleq
    \operatorname{conv}
    \Big(
        \cup_{\vy \in \sB(\vx,\delta)} \partial f(\vy)
    \Big).
\end{align*}
\end{dfn}
We aim to find an approximate Goldstein stationary point~\cite{zhang2020complexity} of the non-smooth non-convex objective.
\begin{dfn}[$(\delta,\epsilon)$-Goldstein stationary point] \label{dfn:Goldsta}
A point $\vx \in \sR^d$ is a $(\delta,\epsilon)$-Goldstein stationary point of $f$ if 
${\rm dist}( \vzero, \partial_\delta f(\vx)) \le \epsilon.$
\end{dfn}

\subsection{Randomized Smoothing}
\label{sec:randomized}
In this section, we introduce randomized smoothing~\cite{duchi2012randomized,yousefian2012stochastic,nesterov2017random} over the unit ball~\cite{hazan2016graduated} and summarize its properties~\cite{lin2022gradient,chen2023faster}.
\begin{dfn}[$\delta$-smoothed surrogate]\label{dfn:f-delta}
For $\delta>0$, define the $\delta$-smoooth surrogate of function $f$ be
\begin{align*}
    f_\delta(\vx) \triangleq \E_{\vu \sim {\rm Unif}(\sB^d)}[f(\vx+\delta \vu)].
\end{align*}
\end{dfn}

\begin{restatable}{prop}{propfdeltasmooth}\label{prop:f-delta-smooth}
Suppose Assumptions \ref{asm:lower} and \ref{asm:lip_function} hold. 
Then for every $\delta>0$, the smoothed function $f_\delta$ satisfies:
\begin{enumerate}
    \item[(a)] $|f_\delta(\vx)-f(\vx)| \le L\delta$ for all $\vx \in \sR^d$.
    \item[(b)] $f_\delta$ is $L$-Lipschitz continuous.
    \item[(c)] $f_\delta$ is differentiable on $\sR^d$.
    \item[(d)] $\nabla f_\delta$ is $(c\sqrt{d}L/\delta)$-Lipschitz continuous for some universal constant $c>0$.
    \item[(e)] $\nabla f_\delta(\vx) \in \partial_\delta f(\vx)$ for all $\vx \in \sR^d$.
\end{enumerate}
\end{restatable}

\subsection{The Online-To-Non-Convex Conversion on the Smoothed Surrogate}
We recall the following O2NC reduction from \citet{cutkosky2023optimal}, which will be used throughout the paper.
We apply the O2NC to the smoothed surrogate $f_{\delta}$ instead of $f$~\cite{kornowski2024algorithm}.
Algorithm \ref{alg:O2NC} presents the update rule.

\begin{algorithm}[htbp]
\caption{O2NC on $f_{\delta}$}\label{alg:O2NC}
\renewcommand{\algorithmicrequire}{\textbf{Input:}}
\renewcommand{\algorithmicensure}{\textbf{Output:}}
\begin{algorithmic}[1]
\Require smoothing radius $\delta>0$, block length $M \in \sN$, number of blocks $K \in \sN$, stepsize $\eta>0$
\State $D=\delta/M$, $T=KM$, $\vx_0=\vzero$
\For{$k=1,\ldots,K$}
\State $\vDelta_{(k-1)M+1}=\vzero$
\For{$m=1,\ldots,M$}
\State $t=(k-1)M+m$
\State Sample $s_t \sim {\rm Unif}[0,1]$
\State 
$\vy_t=\vx_{t-1}+s_t\vDelta_t$
\State 
$\vx_t=\vx_{t-1}+\vDelta_t$
\State Construct a stochastic estimator $\vg_t$ of $\nabla f_\delta(\vy_t)$
\State $\vDelta_{t+1}=\Pi_{\sB(\vzero,D)}(\vDelta_t-\eta \vg_t)$
\EndFor
\EndFor
\State Sample $k_{\rm out} \sim {\rm Unif}\{1,\ldots,K\}$ and return $\bar \vy_{k_{\rm out}}=\frac{1}{M}\sum_{t=(k_{\rm out}-1)M+1}^{k_{\rm out}M}\vy_t$
\end{algorithmic}
\end{algorithm}

The following theorem states the convergence guarantee of O2NC when a suitable estimator of $\nabla f_{\delta}(\vy_t)$ is available. Let $\mathcal{F}_t$ be the filtration that contains all randomness revealed through iteration $t$ and all pre-query randomness for iteration $t+1$; thus, $\mathcal{F}_{t-1}$ includes the interpolation randomness used to form $\vy_t$, but excludes the directions and oracle feedback used to construct $\vg_t$. 

\begin{thm}\label{thm:O2NC-generic}
Suppose Assumptions \ref{asm:estimator}, \ref{asm:lower}, and \ref{asm:lip_function} hold, and the estimator sequence satisfies
\begin{align}
    \label{eq:O2NC_condi}
    \E[\vg_t \mid \mathcal{F}_{t-1}] = \nabla f_\delta(\vy_t),
    \qquad
    \E[\norm{\vg_t}^2] \le G^2
\end{align}
for all $t=1,\ldots, T$. Choose $ \eta=\frac{D}{G\sqrt{M}}$,
then
\begin{align*}
    \E\Big[
        \Big\|
            \frac{1}{M}
            \sum_{t=(k_{\rm out}-1)M+1}^{k_{\rm out}M}
            \nabla f_\delta(\vy_t)
        \Big\|
    \Big]
    \le
    \frac{M(\gamma+\delta L)}{\delta T}
    +
    \frac{2G}{\sqrt{M}},
\end{align*}
where $\gamma\triangleq f(\vx_0)-f^* $.
In particular, it is enough to set $T = \mathcal{O}(G^2(\gamma + \delta L)\delta^{-1}\epsilon^{-3})$
to find the $(\delta,\epsilon)$-Goldstein stationary point of $f$ in expectation.
\end{thm}

\section{Zeroth-Order O2NC for Non-Smooth Non-Convex Decision Dependent Problems}
\label{sec:o2nc-framework}
Although the O2NC framework provides explicit convergence guarantee for non-smooth non-convex problems, it requires a good unbiased estimator of $\nabla f_{\delta}(\vx)$.
When $\xi\sim \Xi$ is decision-independent, the prior zeroth-order methods~\cite{ghadimi2013stochastic,nesterov2017random,lin2022gradient,chen2023faster,kornowski2024algorithm, bach2016highly} use
\begin{align}
    \tilde{\vg}_t \triangleq \frac{d}{2\delta} \left(F(\vy_t+\delta\vu_t;\xi_t) - F(\vy_t-\delta\vu_t;\xi_t)\right)\vu_t
\end{align}
which satisfies Condition~\ref{eq:O2NC_condi} with $\E [\|\tilde{\vg}_t\|^2]\leq dL^2$ and thus leads to the complexity of $\mathcal{O}(dL^2\delta^{-1}\epsilon^{-3})$ through the O2NC framework.
However, $\E [\|\tilde{\vg}_t\|^2]\leq dL^2$ requires the Lipschitz continuity assumption on $F(\vx;\xi)$,
which does not hold in our setting in general.
In addition, for the decision-dependent problems, one cannot access $F(\vy_t+\delta\vu_t;\xi_t)$ and $F(\vy_t-\delta\vu_t;\xi_t)$
under the same $\xi_t$.

In this section, we study two types of zeroth-order estimators for approximating $\nabla f_{\delta}(\vx)$ (see Section~\ref{sec:zo-estimators}). 
We also equip these zeroth-order estimators in the O2NC framework and show the explicit SZO complexity for the non-smooth non-convex decision-dependent problems (Section~\ref{sec:zo-convergence}).

\subsection{Zeroth-Order Gradient Estimators} 
\label{sec:zo-estimators}
We first construct $\vg_t$ by the two-points randomized estimator as in~\cite{hikima2025zeroth}:
\begin{align}
    \label{eq:two_point_estimator}
  \vg_{\text{2pt}}(\vy_t,\vu_t;\xi_t^{+},\xi_t^{-}) \triangleq \frac{d}{2\delta} \left(F(\vy_t+\delta\vu_t;\xi_t^{+}) - F(\vy_t-\delta\vu_t;\xi_t^{-})\right)\vu_t,
\end{align}
where $\vu_t \sim {\rm Unif}(\sS^{d-1})$, $\xi_t^+ \sim \Xi(\vy_t+\delta\vu_t)$, and $\xi_t^- \sim \Xi(\vy_t-\delta\vu_t)$.
\citet{hikima2025zeroth} study how mini-batch version of the two-points estimator \eqref{eq:two_point_estimator} approximates $\nabla f(\vx_t)$ when $f$ is smooth.
In addition, their final convergence result requires setting the per-iteration mini-batch size to $N=\mathcal{O}(d^2\epsilon^{-4})$ per iteration.
The following lemma shows that the two-point estimator can be considered an unbiased stochastic gradient estimator of $\nabla f_{\delta}(\vx_t)$ without smooth conditions or mini-batch requirements.

\begin{lem}\label{lem:zo-var}
Suppose Assumptions \ref{asm:estimator} and \ref{asm:lip_function} hold. Then, we have
\begin{align*}
    \E[\vg_{\text{\rm 2pt}}(\vy_t,\vu_t; \xi_t^{+},\xi_t^{-})\mid \mathcal{F}_{t-1}]
    =
    \nabla f_\delta(\vy_t),
\end{align*}
and
\begin{align*}
    \E\big[
        \norm{\vg_{\text{\rm 2pt}}(\vy_t,\vu_t;\xi_t^{+},\xi_t^{-})}^2
        \mid
        \mathcal{F}_{t-1}
    \big]
    \le
    \frac{d^2\sigma^2}{2\delta^2}
    +
    16\sqrt{2\pi}dL^2.
\end{align*}
\end{lem}

We then construct $\vg_t$ by the one-point estimator. 
The standard one-point estimator used in prior work~\cite{liu2024twotimescale,flaxman2005online},
$\vg_t = \frac{d}{\delta}F(\vx_t+\delta\vu_t;\xi_t)\vu_t$,
is not suitable for our setting, since it leads to a heavy dependence on $\sup_{\vx,\xi} |F(\vx;\xi)|$ and also requires smoothness of $f$.
Inspired by the residual-feedback ideas of~\citet{ye2025onepoint} and \citet{zhang2022new}, we instead use the following one-point estimator, which leverages the stochastic feedback from the previous iteration.
\begin{align}
    \label{eq:one_point_estimator}
    \vg_{\text{1pt}}(\vy_t,\vu_t,\xi_t;\vy_{t-1},\vu_{t-1},\xi_{t-1}) \triangleq \frac{d}{\delta}\left(F(\vy_t+\delta\vu_t;\xi_t) - F(\vy_{t-1}+\delta\vu_{t-1};\xi_{t-1})\right)\vu_t,
\end{align}
where $\vu_t,\vu_{t-1}\sim {\rm Unif}(\sS^{d-1})$, $\xi_t \sim\Xi(\vy_t+\delta\vu_t)$, and $\xi_{t-1}\sim \Xi(\vy_{t-1}+\delta\vu_{t-1})$. 
At iteration $t$, since $F(\vy_{t-1}+\delta \vu_{t-1};\xi_{t-1})$ has already been evaluated in the previous iteration, only the one-point stochastic feedback $F(\vy_t+\delta\vu_t;\xi_t)$ is needed to construct \eqref{eq:one_point_estimator}.
The following lemma shows that $\vg_{\text{1pt}}(\vy_t,\vu_t,\xi_t;\vy_{t-1},\vu_{t-1},\xi_{t-1})$ approximates $\nabla f_{\delta}(\vy_t)$ when $\vy_t$ and $\vy_{t-1}$ are close.

\begin{lem}
    \label{lem:one-point-residual}
Suppose Assumptions \ref{asm:estimator} and \ref{asm:lip_function} hold. Then, we have
\begin{align*}
    \E[ \vg_{\text{\rm 1pt}}(\vy_t,\vu_t,\xi_t;\vy_{t-1},\vu_{t-1},\xi_{t-1})\mid \mathcal{F}_{t-1}]
    =
    \nabla f_\delta(\vy_t),
\end{align*}
and
\begin{align*}
    \E[\norm{ \vg_{\text{\rm 1pt}}(\vy_t,\vu_t,\xi_t;\vy_{t-1},\vu_{t-1},\xi_{t-1})}^2]
    \le
    {\frac{6d^2\sigma^2}{\delta^2}}
    +
  144\sqrt{2\pi}dL^2
    +
    {\frac{6d^2L^2}{\delta^2}}
    \E[\norm{\vy_t-\vy_{t-1}}^2].
\end{align*}
\end{lem}

\begin{rmk}
Compared with the second-order moment term of the two-points estimator in Lemma~\ref{lem:zo-var}, the bound in Lemma \ref{lem:one-point-residual} contains an additional drift term proportional to $\E[\norm{\vy_t-\vy_{t-1}}^2]$.
In the next section, we will show that this drift term can be well-controlled in the O2NC framework.
\end{rmk}

\subsection{The ZO-O2NC Method and Its Convergence Analysis}
\label{sec:zo-convergence}
In this section, we instantiate the O2NC framework with the two-point and one-point estimators in Algorithm~\ref{alg:O2NC-detail}. 
The algorithm is written as a unified template in order to highlight the only difference between the two options, namely, the choice of the gradient estimator $\vg_t$.
For Option~II, the estimator uses the stochastic feedback from the previous iteration and therefore requires an additional initialization step before the first iteration.
To keep the main text concise, Algorithm~\ref{alg:O2NC-detail} presents Option~II in this template form, while a fully specified implementation of Option~II is provided in Appendix~\ref{app:option-ii-alg}.

\begin{algorithm}[htbp]
\caption{{ZO-O2NC}}\label{alg:O2NC-detail}
\renewcommand{\algorithmicrequire}{\textbf{Input:}}
\renewcommand{\algorithmicensure}{\textbf{Output:}}
\begin{algorithmic}[1]
\Require {smoothing radius $\delta>0$, block length $M \in \sN$, number of blocks $K \in \sN$, stepsize $\eta>0$, Option I or II}
\State{$D=\delta/M$, $T=KM$, $\vx_0=\vzero$}
\For{$k=1,\ldots,K$}
\State $\vDelta_{(k-1)M+1}=\vzero$
\For{$m=1,\ldots,M$}
\State $t=(k-1)M+m$
\State Sample $s_t \sim {\rm Unif}[0,1]$
\State $\vy_t=\vx_{t-1}+s_t\vDelta_t$
\State $\vx_t=\vx_{t-1}+\vDelta_t$
\State $\vg_t = {\begin{cases}
     &\vg_{\text{2pt}}(\vy_t,\vu_t,\xi_t^{+},\xi_t^{-})~~\text{by~\eqref{eq:two_point_estimator}}~~~~~~~~~~~~~~~~~~~~~~~~~~~\text{[Option I]}\\
     & \vg_{\text{1pt}}(\vy_t,\vu_t,\xi_t;\vy_{t-1},\vu_{t-1},\xi_{t-1})~~\text{by~\eqref{eq:one_point_estimator}}~~~~~~\text{[Option II]}
\end{cases}}
$
\State $\vDelta_{t+1}=\Pi_{\sB(\vzero,D)}(\vDelta_t-\eta \vg_t)$ 
\EndFor
\EndFor
\State Sample $k_{\rm out} \sim {\rm Unif}\{1,\ldots,K\}$ and return $\bar \vy_{k_{\rm out}}=\frac{1}{M}\sum_{t=(k_{\rm out}-1)M+1}^{k_{\rm out}M}\vy_t$
\end{algorithmic}
\end{algorithm}

We first show that the second moment bound of the one-point estimator can 
be well controlled in the O2NC framework since $\Delta_t$ is always bounded by $D$, which in turn controls $\|\vy_t-\vy_{t-1}\|$.
The following lemma bounds the second moment of $\vg_t$ in O2NC with Option II.
\begin{lem}\label{lem:one-point-o2nc-moment}
Suppose Assumptions \ref{asm:estimator} and \ref{asm:lip_function} hold, and Algorithm~\ref{alg:O2NC-detail} is run with Option II. Then, for every $t \ge 1$,
\begin{align*}
    \E[\norm{\vg_t}^2]
    \le
{\frac{6d^2\sigma^2}{\delta^2}}
    +
   {385d^2L^2}.
\end{align*}
\end{lem}
Now, we are ready to present the formal SZO complexity of our ZO-O2NC method for solving non-smooth non-convex decision-dependent problems.
\begin{thm}\label{thm:O2NC-NSNC}
Suppose Assumptions \ref{asm:estimator}, \ref{asm:lower}, and \ref{asm:lip_function} hold, and define
\begin{align*}
    \bar \vg_{k_{\rm out}}
    \triangleq
    \frac{1}{M}
    \sum_{t=(k_{\rm out}-1)M+1}^{k_{\rm out}M}
    \nabla f_\delta(\vy_t),
    \qquad
    \gamma
    \triangleq
    f(\vx_0)-f^*
\end{align*}
Then the following statements hold.
\begin{enumerate}
    \item[(a)] If Algorithm~\ref{alg:O2NC-detail} is run with Option I and
    \begin{align*}
        G_{\rm I}^2
        \triangleq
        \frac{d^2\sigma^2}{2\delta^2}
        +
        \aireplace{16\sqrt{2\pi}dL^2}{2dL^2},~~
        M
        =
        \left\lceil
            \frac{16G_{\rm I}^2}{\epsilon^2}
        \right\rceil,~~
\eta = \frac{D}{G_{\rm I}\sqrt{M}}, ~~ T
        =
        M
        \left\lceil
            \frac{2(\gamma+\delta L)}{\delta\epsilon}
        \right\rceil,
    \end{align*}
then $\bar \vy_{k_{\rm out}}$ satisfies $ \E[\norm{\bar \vg_{k_{\rm out}}}]
        \le
        \epsilon.$
This means that Algorithm~\ref{alg:O2NC-detail} can find the $(\delta,\epsilon)$-Goldstein stationary point of $f(\cdot)$ with SZO complexity of 
    \begin{align*}
        N_{\rm SZO}
        =
        2T = \mathcal{O}\left(
           d^2\gamma \sigma^2\delta^{-3}\epsilon^{-3} + d^2 L\sigma^2\delta^{-2}\epsilon^{-3} + dL^2\gamma\delta^{-1}\epsilon^{-3} +dL^3\epsilon^{-3}.
            \right).  
    \end{align*}
    \item[(b)] If Algorithm~\ref{alg:O2NC-detail} is run with Option II and 
    \begin{align*}
        G_{\rm II}^2
        \triangleq
        {\frac{6d^2\sigma^2}{\delta^2}}
        +
       {\aireplace{385d^2L^2}{42d^2L^2}},
       ~~~
        M
        =
        \left\lceil
            \frac{16G_{\rm II}^2}{\epsilon^2}
        \right\rceil,
       ~~~
                \eta = \frac{D}{G_{\rm II}\sqrt{M}},
           ~~~
        T
        =
        M
        \left\lceil
            \frac{2(\gamma+\delta L)}{\delta\epsilon}
        \right\rceil,
    \end{align*}
    then $\bar \vy_{k_{\rm out}}$ satisfies $  \E[\norm{\bar \vg_{k_{\rm out}}}]
        \le
        \epsilon.$
        This means that Algorithm~\ref{alg:O2NC-detail} with Option II can find the $(\delta,\epsilon)$-Goldstein stationary point of $f(\cdot)$ with SZO complexity of 
    \begin{align*}
        N_{\rm SZO}
        =
        T+1 = \mathcal{O}\left(
           d^2\gamma \sigma^2\delta^{-3}\epsilon^{-3} + d^2 L\sigma^2\delta^{-2}\epsilon^{-3} + d^2L^2\gamma\delta^{-1}\epsilon^{-3} +d^2L^3\epsilon^{-3}
            \right).  
    \end{align*}
  \end{enumerate}
\end{thm}
Theorem~\ref{thm:O2NC-NSNC} demonstrates that one only needs to query constant or even a single SZO to find the $(\delta,\epsilon)$-Goldstein stationary point for nonsmooth nonconvex distribution dependent problems.
Moreover, the one-point estimator shares the same dominant term in SZO complexity as the two-points estimator.

\section{Extension to Smooth Non-Convex Problems}
\label{sec:smooth-extension}
We now extend the guarantees of the ZO-O2NC method to smooth non-convex problems.
We impose the following additional assumptions on $f$.
\begin{asm}\label{asm:grad-lip}
The gradient of $f$ is $L_g$-Lipschitz continuous, that is, for all $\vx,\vy\in\sR^d$, we have
\begin{align*}
    \norm{\nabla f(\vx)-\nabla f(\vy)} \le L_g \norm{\vx-\vy}.
\end{align*}
\end{asm}
\begin{asm}\label{asm:Hess-lip}
The Hessian of $f$ is $L_H$-Lipschitz continuous, that is, for all $\vx,\vy\in\sR^d$, we have
\begin{align*}
    \norm{\nabla^2 f(\vx)-\nabla^2 f(\vy)} \le L_H \norm{\vx-\vy}.
\end{align*}
\end{asm}
We use the standard notion of $\epsilon$-stationarity.
\begin{dfn}[$\epsilon$-stationary point]
A point $\vx \in \sR^d$ is $\epsilon$-stationary if $ \norm{\nabla f(\vx)} \le \epsilon.$
\end{dfn}
\citet[Propositions 14 and 15]{cutkosky2023optimal} relate
$
\Big\|
\frac{1}{M}\sum_{i=1}^{M}\nabla f(\vy_i)
\Big\|
$ and $
\Big\|\nabla f\big(\frac{1}{M}\sum_{i=1}^M\vy_i\big)\Big\|
$
under gradient Lipschitzness and Hessian Lipschitzness.
However, Theorem~\ref{thm:O2NC-NSNC} only controls $\|
\frac{1}{M}\sum_{i=1}^{M}\nabla f_{\delta}(\vy_i)
\|$
instead.
The next proposition bridges this gap by reducing the certificate based on the smoothed gradients to stationarity of the original objective.

\begin{prop}\label{prop:smooth-reduction}
Let $\vy_1,\ldots,\vy_M \in \sR^d$ satisfy $\norm{\vy_i-\bar{\vy}} \le \delta$, where $\bar{\vy}=\frac{1}{M}\sum_{i=1}^M \vy_i$.
Then the following statements hold:
\begin{enumerate}
    \item[(a)] Under Assumption \ref{asm:grad-lip},
    \begin{align*}
        \norm{\nabla f(\bar{\vy})}
        \le
        \left\|
            \frac{1}{M}\sum_{i=1}^{M}\nabla f_\delta(\vy_i)
        \right\|
        +
        2L_g\delta.
    \end{align*}
    \item[(b)] Under Assumption \ref{asm:Hess-lip},
    \begin{align*}
        \norm{\nabla f(\bar{\vy})}
        \le
        \left\|
            \frac{1}{M}\sum_{i=1}^{M}\nabla f_\delta(\vy_i)
        \right\|
        +
        L_H\delta^2.
    \end{align*}
\end{enumerate}
\end{prop}
With Proposition~\ref{prop:smooth-reduction}, we can now derive SZO complexity bounds for smooth and Hessian-Lipschitz decision-dependent problems.
\begin{thm}\label{thm:smooth-gradient}
Suppose Assumptions \ref{asm:estimator}, \ref{asm:lower}, \ref{asm:lip_function}, and \ref{asm:grad-lip} hold. Let $ \delta = \frac{\epsilon}{4L_g}$, 
then Algorithm~\ref{alg:O2NC-detail} with Option I or Option II can find an $\epsilon$-stationary point using $N_{\rm SZO} = \gO\left(d^2 \epsilon^{-6}\right)$ queries.
\end{thm}
\begin{thm}\label{thm:smooth-hessian}
Suppose Assumptions \ref{asm:estimator}, \ref{asm:lower}, \ref{asm:lip_function}, and \ref{asm:Hess-lip} hold. Let $ \delta = \sqrt{\frac{\epsilon}{2L_H}}$, then 
Algorithm~\ref{alg:O2NC-detail} with Option I or Option II can find an $\epsilon$-stationary point using $ N_{\rm SZO} = \gO\left(d^2\epsilon^{-9/2}\right)$ queries.

\end{thm}
\paragraph{Comparison with \citet{hikima2025zeroth}.}
\citet{hikima2025zeroth} also studied these two settings.
Compared with their method, ZO-O2NC has two main advantages:
\begin{itemize}
    \item Their method requires a large mini-batch, with batch size $N=\mathcal{O}(d^2\epsilon^{-4})$ per iteration, whereas ZO-O2NC uses only a constant number of SZO queries per iteration, and even only one fresh query under Option II.
    \item Under Hessian Lipschitzness, their method requires $\mathcal{O}(d^{2}\epsilon^{-5})$ SZO complexity, while ZO-O2NC only requires $\mathcal{O}(d^2\epsilon^{-9/2})$ queries, improving the dependence on $\epsilon$ by a factor of $\epsilon^{-1/2}$.
\end{itemize}
\begin{rmk}
 In Appendix~\ref{app:sgd-baseline}, we also prove that SGD with mini-batch two-point estimator~\cite{hikima2025zeroth} can find the $(\delta,\epsilon)$-Goldstein stationary point by directly analyzing the descent of $f_{\delta}(\cdot)$.
Their method has an SZO complexity of $\mathcal{O}(d^{5/2}\delta^{-3}\epsilon^{-4})$, which is worse than ZO-O2NC (Algorithm~\ref{alg:O2NC-detail}).
\end{rmk}

%% file: Additional_related.tex

%% file: appendix.tex
\section{The Proof of Section~\ref{sec:preliminaries}}
\label{app:proof-sec2}

\subsection{The Proof of Proposition \ref{prop:f-delta-smooth}}
\label{app:proof-prop-f-delta-smooth}
\begin{proof}
Part (a) follows from Assumption \ref{asm:lip_function}:
\begin{align*}
    |f_\delta(\vx)-f(\vx)|
    &=
    \left|
        \E_{\vu \sim {\rm Unif}(\sB^d)}
        \big[
            f(\vx+\delta\vu)-f(\vx)
        \big]
    \right| \\
    &\le
    \E\big[
        |f(\vx+\delta\vu)-f(\vx)|
    \big]
    \le
    L\delta \E[\norm{\vu}]
    \le
    L\delta.
\end{align*}
Similarly, for any $\vx,\vy \in \sR^d$,
\begin{align*}
    |f_\delta(\vx)-f_\delta(\vy)|
    &=
    \left|
        \E_{\vu \sim {\rm Unif}(\sB^d)}
        \big[
            f(\vx+\delta\vu)-f(\vy+\delta\vu)
        \big]
    \right| \\
    &\le
    \E\big[
        |f(\vx+\delta\vu)-f(\vy+\delta\vu)|
    \big]
    \le
    L\norm{\vx-\vy},
\end{align*}
which proves part (b).

Parts (c)--(e) are standard consequences of randomized smoothing over Euclidean balls for Lipschitz functions; see, for example, \citet{duchi2012randomized,zhang2020complexity}. In particular, $f_\delta$ is differentiable on $\sR^d$, $\nabla f_\delta$ is $(c\sqrt{d}L/\delta)$-Lipschitz continuous for some universal constant $c>0$, and $\nabla f_\delta(\vx) \in \partial_\delta f(\vx)$ for all $\vx \in \sR^d$.
\end{proof}

\subsection{The Proof of Theorem \ref{thm:O2NC-generic}}
\label{app:proof-thm-o2nc-generic}
\begin{proof}
For a sequence $\{\vDelta_t\}_{t=1}^{T}$ and an estimator sequence $\{\vg_t\}_{t=1}^{T}$, define the block regret
\begin{align*}
    {\rm Reg}_{K,M}
    \triangleq
    \sup_{\vu_1,\ldots,\vu_K \in \sB(\vzero,D)}
    \sum_{k=1}^K
    \sum_{t=(k-1)M+1}^{kM}
    \inner{\vg_t}{\vDelta_t-\vu_k}.
\end{align*}
For each block $k \in \{1,\ldots,K\}$, define
\begin{align*}
    \bar \vg_k
    \triangleq
    \frac{1}{M}\sum_{t=(k-1)M+1}^{kM}\nabla f_\delta(\vy_t).
\end{align*}
Let $\zeta_t\triangleq\vg_t-\nabla f_\delta(\vy_t)$. Then $\E[\zeta_t\mid \mathcal{F}_{t-1}]=\vzero$, and
\begin{align*}
    \E[\norm{\zeta_t}^2]
    =
    \E[\norm{\vg_t}^2]-\E[\norm{\nabla f_\delta(\vy_t)}^2]
    \le G^2.
\end{align*}
By \citet[Theorem 4]{cutkosky2023optimal},
\begin{align*}
   & \frac{1}{K}
    \sum_{k=1}^K
    \E
    \left[
        \left\|
            \frac{1}{M}\sum_{t=(k-1)M+1}^{kM}\nabla f_\delta(\vy_t)
        \right\|
    \right]
    \\
    &\le
    \frac{f_\delta(\vx_0)-f_\delta^*}{KDM}
    +
    \frac{\E[{\rm Reg}_{K,M}]}{KDM}
    +
    \frac{1}{K}
    \sum_{k=1}^K
    \E
    \left[
        \left\|
            \frac{1}{M}\sum_{t=(k-1)M+1}^{kM}\zeta_t
        \right\|
    \right].
\end{align*}
Since $\{\zeta_t\}_{t=1}^T$ is a martingale difference sequence,
\begin{align*}
    \E
    \left[
        \left\|
            \frac{1}{M}\sum_{t=(k-1)M+1}^{kM}\zeta_t
        \right\|
    \right]
    \le
    \left(
        \E
        \left[
            \left\|
                \frac{1}{M}\sum_{t=(k-1)M+1}^{kM}\zeta_t
            \right\|^2
        \right]
    \right)^{1/2}
    \le
    \frac{G}{\sqrt{M}}.
\end{align*}
Moreover, the standard OGD bound yields
\begin{align*}
    \E[{\rm Reg}_{K,M}]
    \le
    \frac{KD^2}{2\eta}
    +
    \frac{\eta KMG^2}{2}
    =
    KDG\sqrt{M}.
\end{align*}
Substituting this bound and $\eta=D/(G\sqrt{M})$ into the preceding display gives
\begin{align*}
    \frac{1}{K}
    \sum_{k=1}^K
    \E[\norm{\bar \vg_k}]
    =
    \gO
    \left(
        \frac{f_\delta(\vx_0)-f_\delta^*}{KDM}
        +
        \frac{2G}{\sqrt{M}}
    \right).
\end{align*}
Since $k_{\rm out}$ is sampled uniformly from $\{1,\ldots,K\}$,
\begin{align*}
    \E[\norm{\bar \vg_{k_{\rm out}}}]
    =
    \frac{1}{K}\sum_{k=1}^K \E[\norm{\bar \vg_k}]
    =
    \gO
    \left(
        \frac{f_\delta(\vx_0)-f_\delta^*}{KDM}
        +
        \frac{2G}{\sqrt{M}}
    \right).
\end{align*}
Finally, Assumption \ref{asm:lower} and Proposition \ref{prop:f-delta-smooth}(a) imply
\begin{align*}
    f_\delta(\vx_0)-f_\delta^*
    \le
    f(\vx_0)-f^*+L\delta
    \le
    \gamma+\delta L.
\end{align*}
Using $D=\delta/M$ and $T=KM$ gives
\begin{align*}
    \E[\norm{\bar \vg_{k_{\rm out}}}]
    \le
    \frac{M(\gamma+\delta L)}{\delta T}
    +
    \frac{2G}{\sqrt{M}}.
\end{align*}
Therefore, by choosing $M$ and $T$ as in the theorem statement, we obtain
\begin{align*}
    \E[\norm{\bar \vg_{k_{\rm out}}}]
    \le
    \epsilon.
\end{align*}
Moreover,
\begin{align*}
    T
    =
    M
    \left\lceil
        \frac{2(\gamma+\delta L)}{\delta\epsilon}
    \right\rceil
    =
    \gO
    \left(
        \frac{(\gamma+\delta L)G^2}{\delta\epsilon^3}
    \right).
\end{align*}
\end{proof}

\section{The Proof of Section~\ref{sec:o2nc-framework}}
\label{app:proof-sec3}

\begin{lem}[Variance bound for Lipschitz functions on the sphere]
\label{lem:sphere-variance}
Let $h:\sS^{d-1}\to \sR$ be $\rho$-Lipschitz, and let $\vu \sim {\rm Unif}(\sS^{d-1})$. Then
\begin{align*}
    \Var(h(\vu))
    \le
    \frac{16\sqrt{2\pi}\rho^2}{d}.
\end{align*}
\end{lem}
\begin{proof}
If $d=1$, then $\sS^0=\{-1,1\}$ and $\vu$ is uniform on this two-point set. Since $h$ is $\rho$-Lipschitz,
\begin{align*}
    |h(1)-h(-1)| \le \rho |1-(-1)| = 2\rho.
\end{align*}
Therefore,
\begin{align*}
    \Var(h(\vu))
    =
    \frac{(h(1)-h(-1))^2}{4}
    \le
    \rho^2
    \le
    \frac{16\sqrt{2\pi}\rho^2}{d}.
\end{align*}
\noindent Suppose now that $d\ge2$. Following the proof of Lemma D.1 in \citet{lin2022gradient}, one has the concentration inequality~\cite{wainwright2019high}:
\begin{align*}
    \Pr\Big(
        \big|h(\vu)-\E[h(\vu)]\big|
        \ge
        \alpha
    \Big)
    \le
    2\sqrt{2\pi}
    \exp\left(
        -\frac{\alpha^2 d}{8\rho^2}
    \right),
    \qquad
    \forall \alpha \ge 0.
\end{align*}
Therefore,
\begin{align*}
    \Var(h(\vu))
    &=
    \E\Big[
        \big(h(\vu)-\E[h(\vu)]\big)^2
    \Big] \\
    &=
    \int_0^\infty
    \Pr\Big(
        \big(h(\vu)-\E[h(\vu)]\big)^2
        \ge
        \alpha
    \Big)\,d\alpha \\
    &=
    \int_0^\infty
    \Pr\Big(
        \big|h(\vu)-\E[h(\vu)]\big|
        \ge
        \sqrt{\alpha}
    \Big)\,d\alpha \\
    &\le
    2\sqrt{2\pi}
    \int_0^\infty
    \exp\left(
        -\frac{\alpha d}{8\rho^2}
    \right)\,d\alpha \\
    &=
    \frac{16\sqrt{2\pi}\rho^2}{d}.
\end{align*}
\end{proof}

\subsection{The Proof of Lemma \ref{lem:zo-var}}
\label{app:proof-lem-zo-var}
\begin{proof}
Write
\begin{align*}
    \vg_{\text{2pt},t} \triangleq \vg_{\text{2pt}}(\vy_t,\vu_t,\xi_t^{+},\xi_t^{-}),
    \qquad
    S_t \triangleq f(\vy_t+\delta\vu_t)-f(\vy_t-\delta\vu_t).
\end{align*}
The expectation identity is standard for randomized smoothing; see, for example, \citet{duchi2012randomized,shamir2017optimal}. Since $\vy_t$ is $\mathcal{F}_{t-1}$-measurable,
\begin{align*}
    \E[\vg_{\text{2pt},t}\mid \mathcal{F}_{t-1}]
    =
    \frac{d}{2\delta}
    \E_{\vu \sim {\rm Unif}(\sS^{d-1})}
    \big[
        \big(
            f(\vy_t+\delta\vu)-f(\vy_t-\delta\vu)
        \big)\vu
    \big]
    =
    \nabla f_\delta(\vy_t).
\end{align*}
Equivalently,
\begin{align*}
    \E[S_t \vu_t \mid \mathcal{F}_{t-1}] = \frac{2\delta}{d}\nabla f_\delta(\vy_t).
\end{align*}
Let
\begin{align*}
    Z_t^+ \triangleq F(\vy_t+\delta\vu_t;\xi_t^+) - f(\vy_t+\delta\vu_t),
    \qquad
    Z_t^- \triangleq F(\vy_t-\delta\vu_t;\xi_t^-) - f(\vy_t-\delta\vu_t).
\end{align*}
Since $\norm{\vu_t}=1$,
\begin{align}
    \label{eq:E_2_point_est_mid}
    \begin{split}
   & \E\big[
        \norm{\vg_{\text{2pt},t}}^2
        \mid
        \mathcal{F}_{t-1}
    \big]\\
    &=
    \frac{d^2}{4\delta^2}
    \E\big[
        |Z_t^+-Z_t^-+S_t|^2
        \mid
        \mathcal{F}_{t-1}
    \big] \\
    &=
    \frac{d^2}{4\delta^2}
    \E\big[ |Z_t^+-Z_t^-|^2 \mid \mathcal{F}_{t-1} \big]
    +
    \frac{d^2}{4\delta^2}
    \E\big[ |S_t|^2 \mid \mathcal{F}_{t-1} \big] 
    +
    \frac{d^2}{2\delta^2}
    \E\big[ (Z_t^+-Z_t^-)S_t \mid \mathcal{F}_{t-1} \big].
    \end{split}
\end{align}
We claim that the last term is zero. Indeed, $S_t$ is measurable with respect to $(\mathcal{F}_{t-1},\vu_t)$, while
\begin{align*}
    \E[Z_t^+\mid \mathcal{F}_{t-1},\vu_t] = 0,
    \qquad
    \E[Z_t^-\mid \mathcal{F}_{t-1},\vu_t] = 0.
\end{align*}
Therefore,
\begin{align*}
    \E\big[ (Z_t^+-Z_t^-)S_t \mid \mathcal{F}_{t-1} \big]
    =
    \E\Big[
        S_t
        \E[ Z_t^+-Z_t^- \mid \mathcal{F}_{t-1},\vu_t ]
        \,\Big|\,
        \mathcal{F}_{t-1}
    \Big]
    =
    0.
\end{align*}
Hence
\begin{align*}
    \E\big[
        \norm{\vg_{\text{2pt},t}}^2
        \mid
        \mathcal{F}_{t-1}
    \big]
    =
    \frac{d^2}{4\delta^2}
    \E\big[ |Z_t^+-Z_t^-|^2 \mid \mathcal{F}_{t-1} \big]
    +
    \frac{d^2}{4\delta^2}
    \E\big[ |S_t|^2 \mid \mathcal{F}_{t-1} \big].
\end{align*}
Assumption \ref{asm:estimator} gives
\begin{align*}
    \E\big[ |Z_t^+-Z_t^-|^2 \mid \mathcal{F}_{t-1} \big]
    =
    \E\big[ |Z_t^+|^2+|Z_t^-|^2 \mid \mathcal{F}_{t-1} \big]
    \le
    2\sigma^2.
\end{align*}
Finally, we bound the term $\E[|S_t|^2|\mathcal{F}_{t-1}]$. We define
\begin{align*}
    \psi_t(\vu)
    \triangleq
    f(\vy_t+\delta\vu)-f(\vy_t-\delta\vu),
\end{align*}
where  $\vu \in \sS^{d-1}$.
Then $\psi_t$ is odd, so $\E_{\vu \sim {\rm Unif}(\sS^{d-1})}[\psi_t(\vu)]=0$. Moreover, Assumption \ref{asm:lip_function} implies
\begin{align*}
    |\psi_t(\vu)-\psi_t(\vv)|
    \le
    |f(\vy_t+\delta\vu)-f(\vy_t+\delta\vv)|
    +
    |f(\vy_t-\delta\vu)-f(\vy_t-\delta\vv)|
    \le
    2\delta L \norm{\vu-\vv},
\end{align*}
so $\psi_t$ is $(2\delta L)$-Lipschitz on $\sS^{d-1}$. Therefore, Lemma \ref{lem:sphere-variance} yields
\begin{align*}
    \E\big[ |S_t|^2 \mid \mathcal{F}_{t-1} \big]
    =
    \Var(\psi_t(\vu_t)\mid\mathcal{F}_{t-1})
    \le
    \frac{64\sqrt{2\pi}\delta^2L^2}{d}.
\end{align*}
Putting these bound together in \eqref{eq:E_2_point_est_mid}, we obtain
\begin{align*}
    \E\big[ \norm{\vg_{\text{2pt},t}}^2 \mid \mathcal{F}_{t-1} \big]
    \le
    \frac{d^2\sigma^2}{2\delta^2}
    +
    16\sqrt{2\pi}dL^2.
\end{align*}
\end{proof}

\subsection{The Proof of Lemma \ref{lem:one-point-residual}}
\label{app:proof-lem-one-point-residual}
\begin{proof}
We write
\begin{align*}
    \vg_{\text{1pt},t}
    \triangleq
    \vg_{\text{1pt}}(\vy_t,\vu_t,\xi_t;\vy_{t-1},\vu_{t-1},\xi_{t-1}),
\end{align*}
and let
\begin{align*}
    &F_t^+ \triangleq F(\vy_t+\delta\vu_t;\xi_t),
    \qquad
    F_{t-1}^+ \triangleq F(\vy_{t-1}+\delta\vu_{t-1};\xi_{t-1}), \\
    &f_t^+ \triangleq f(\vy_t+\delta\vu_t),
    \qquad
    f_{t-1}^+ \triangleq f(\vy_{t-1}+\delta\vu_{t-1}).
\end{align*}
Since $\vu_t$ and $\xi_t$ are independent of $\mathcal{F}_{t-1}$, with $\E[\vu_t]=\vzero$, we have
\begin{align*}
    \E[\vg_{\text{1pt},t} \mid \mathcal{F}_{t-1}]
    &=
    \frac{d}{\delta}
    \E\left[
        F_t^+ \vu_t - F_{t-1}^+\vu_t
        \mid
        \mathcal{F}_{t-1}
    \right] \\
    &=
    \frac{d}{\delta}
    \E\left[
        f_t^+ \vu_t
        \mid
        \mathcal{F}_{t-1}
    \right] \\
    &=
    \frac{d}{\delta}
    \E_{\vu \sim {\rm Unif}(\sS^{d-1})}
    \left[
        f(\vy_t+\delta\vu)\vu
    \right]
    =
    \nabla f_\delta(\vy_t),
\end{align*}
where the last identity is the standard randomized-smoothing gradient formula.

We define
\begin{align*}
    Z_t
    \triangleq
    F_t^+-f_t^+,
    \qquad
    Z_{t-1}
    \triangleq
    F_{t-1}^+-f_{t-1}^+.
\end{align*}
Since $\norm{\vu_t}=1$,
\begin{align*}
    \E[\norm{\vg_{\text{1pt},t}}^2]
    &=
    \frac{d^2}{\delta^2}
    \E\left[
        |Z_t-Z_{t-1}+f_t^+-f_{t-1}^+|^2
    \right].
\end{align*}
Let
\begin{align*}
    A_t \triangleq f_t^+-f_{t-1}^+.
\end{align*}
Using $(a+b+c)^2 \le 3a^2+3b^2+3c^2$, we obtain
\begin{align*}
    \E[\norm{\vg_{\text{1pt},t}}^2]
    &=
    \frac{d^2}{\delta^2}
    \E\left[
        |Z_t-Z_{t-1}+A_t|^2
    \right] \\
    &\le
    \frac{d^2}{\delta^2}
    \left(
        3\E[|Z_t|^2]
        +
        3\E[|Z_{t-1}|^2]
        +
        3\E[|A_t|^2]
    \right).
\end{align*}

Assumption \ref{asm:estimator} yields
\begin{align*}
    \E[|Z_t|^2] \le \sigma^2,
    \qquad
    \E[|Z_{t-1}|^2] \le \sigma^2.
\end{align*}
Hence
\begin{align*}
    \E[\norm{\vg_{\text{1pt},t}}^2]
    \le
    \frac{6d^2\sigma^2}{\delta^2}
    +
    \frac{3d^2}{\delta^2}
    \E\left[
        |f_t^+-f_{t-1}^+|^2
    \right].
\end{align*}

Define the sphere-smoothed mean
\begin{align*}
    f_{\delta}^{\sS}(\vy)
    \triangleq
    \E_{\vu \sim {\rm Unif}(\sS^{d-1})}
    \big[
        f(\vy+\delta\vu)
    \big].
\end{align*}
Then
\begin{align*}
    \E\left[
        |f_t^+-f_{t-1}^+|^2
        \mid
        \mathcal{F}_{t-1}
    \right]
    &=
    \Var(f_t^+\mid \mathcal{F}_{t-1})
    +
    \big(
        f_{\delta}^{\sS}(\vy_t)-f_{t-1}^+
    \big)^2 \\
    &\le
    \Var(f_t^+\mid \mathcal{F}_{t-1})
    +
    2\big(
        f_{\delta}^{\sS}(\vy_t)-f_{\delta}^{\sS}(\vy_{t-1})
    \big)^2
    +
    2\big(
        f_{\delta}^{\sS}(\vy_{t-1})-f_{t-1}^+
    \big)^2.
\end{align*}
Taking expectations gives
\begin{align*}
    &\E\left[
        |f_t^+-f_{t-1}^+|^2
    \right]\\
    &\le
    \E\big[\Var(f_t^+\mid \mathcal{F}_{t-1})\big] 
    +
    2\E\left[
        \big(
            f_{\delta}^{\sS}(\vy_t)-f_{\delta}^{\sS}(\vy_{t-1})
        \big)^2
    \right] 
    +
    2\E\left[
        \big(
            f_{\delta}^{\sS}(\vy_{t-1})-f_{t-1}^+
        \big)^2
    \right].
\end{align*}
Now, for a fixed $\vy \in \sR^d$, the map
\begin{align*}
    \phi_{\vy}(\vu)
    \triangleq
    f(\vy+\delta\vu),
    \qquad
    \vu \in \sS^{d-1},
\end{align*}
is $(\delta L)$-Lipschitz on $\sS^{d-1}$. Hence, Lemma \ref{lem:sphere-variance} yields
\begin{align*}
    \Var(\phi_{\vy}(\vu))
    \le
    \frac{16\sqrt{2\pi}\delta^2L^2}{d}.
\end{align*}
Therefore,
\begin{align*}
    \E\big[\Var(f_t^+\mid \mathcal{F}_{t-1})\big]
    \le
    \frac{16\sqrt{2\pi}\delta^2L^2}{d},
    \qquad
    \E\left[
        \big(
            f_{\delta}^{\sS}(\vy_{t-1})-f_{t-1}^+
        \big)^2
    \right]
    \le
    \frac{16\sqrt{2\pi}\delta^2L^2}{d}.
\end{align*}
Moreover, Assumption \ref{asm:lip_function} implies that $f_{\delta}^{\sS}$ is $L$-Lipschitz such that
\begin{align*}
    \big|
        f_{\delta}^{\sS}(\vy_t)-f_{\delta}^{\sS}(\vy_{t-1})
    \big|
    \le
    L\norm{\vy_t-\vy_{t-1}}.
\end{align*}
Substituting these bounds into the preceding display yields
\begin{align*}
    \E\left[
        |f_t^+-f_{t-1}^+|^2
    \right]
    \le
    \frac{48\sqrt{2\pi}\delta^2L^2}{d}
    +
    2L^2\E[\norm{\vy_t-\vy_{t-1}}^2].
\end{align*}
Consequently,
\begin{align*}
    \E[\norm{\vg_{\text{1pt},t}}^2]
    \le
    \frac{6d^2\sigma^2}{\delta^2}
    +
    144\sqrt{2\pi}dL^2
    +
    \frac{6d^2L^2}{\delta^2}
    \E[\norm{\vy_t-\vy_{t-1}}^2].
\end{align*}
\end{proof}

\subsection{The Proof of Lemma \ref{lem:one-point-o2nc-moment}}
\label{app:proof-lem-one-point-o2nc-moment}
\begin{proof}
By Lemma \ref{lem:one-point-residual}, it remains to bound $\norm{\vy_t-\vy_{t-1}}$. Since $\vDelta_1=\vzero$, we have $\vy_1=\vx_0=\vy_0$. For $t \ge 2$,
\begin{align*}
    \vy_t-\vy_{t-1}
    =
    s_t\vDelta_t+(1-s_{t-1})\vDelta_{t-1},
\end{align*}
and therefore
\begin{align*}
    \norm{\vy_t-\vy_{t-1}}
    \le
    \norm{\vDelta_t}+\norm{\vDelta_{t-1}}
    \le
    2D.
\end{align*}
\begin{align*}
    \E[\norm{\vg_t}^2]
    \le
    \frac{6d^2\sigma^2}{\delta^2}
    +
    144\sqrt{2\pi}dL^2
    +
    \frac{24d^2L^2D^2}{\delta^2}.
\end{align*}
Since $D=\delta/M$, $M\ge1$, and $d\ge1$, we further have
\begin{align*}
    144\sqrt{2\pi}dL^2+\frac{24d^2L^2D^2}{\delta^2}
    \le
    144\sqrt{2\pi}d^2L^2+24d^2L^2
    \le
    385d^2L^2.
\end{align*}
\end{proof}

\subsection{The Proof of Theorem \ref{thm:O2NC-NSNC}}
\label{app:proof-thm-o2nc-nsnc}
\begin{proof}
Define
\begin{align*}
    \bar \vg_{k_{\rm out}}
    \triangleq
    \frac{1}{M}
    \sum_{t=(k_{\rm out}-1)M+1}^{k_{\rm out}M}
    \nabla f_\delta(\vy_t).
\end{align*}

\textbf{Option I.}
Lemma \ref{lem:zo-var} gives
\begin{align*}
    \E\big[
        \norm{\vg_{\text{\rm 2pt}}(\vy_t,\vu_t,\xi_t^{+},\xi_t^{-})}^2
    \big]
    \le
    G_{\rm I}^2
    \triangleq
    \frac{d^2\sigma^2}{2\delta^2}
    +
    \aireplace{16\sqrt{2\pi}dL^2}{2dL^2}.
\end{align*}
Applying Theorem \ref{thm:O2NC-generic} with $G=G_{\rm I}$ yields
\begin{align*}
    \E[\norm{\bar \vg_{k_{\rm out}}}]
    \le
    \frac{M(\gamma+\delta L)}{\delta T}
    +
    \frac{2G_{\rm I}}{\sqrt{M}}.
\end{align*}
Thus, by choosing $M$ and $T$ as in part (a) of the theorem, we have
\begin{align*}
    \E[\norm{\bar \vg_{k_{\rm out}}}]
    \le
    \epsilon.
\end{align*}
Since each iteration uses two function-value queries,
\begin{align*}
    N_{\rm SZO}
    =
    2T.
\end{align*}
Using the choice of $T$ and the bound $M=\mathcal{O}(G_{\rm I}^2\epsilon^{-2})$, we obtain
\begin{align*}
    N_{\rm SZO}
    &=
    2T \\
    &=
    \mathcal{O}
    \left(
        \frac{(\gamma+\delta L)G_{\rm I}^2}{\delta\epsilon^3}
    \right) \\
    &=
    \mathcal{O}
    \left(
        \frac{\gamma+\delta L}{\delta\epsilon^3}
        \left(
            \frac{d^2\sigma^2}{\delta^2}
            +
            dL^2
        \right)
    \right) \\
    &=
    \mathcal{O}
    \left(
        d^2\gamma \sigma^2\delta^{-3}\epsilon^{-3}
        +
        d^2 L\sigma^2\delta^{-2}\epsilon^{-3}
        +
        dL^2\gamma\delta^{-1}\epsilon^{-3}
        +
        dL^3\epsilon^{-3}
    \right).
\end{align*}
To obtain a $(\delta,\epsilon)$-Goldstein stationary point, set $\widehat{\delta} \triangleq \delta/2$ and run Algorithm \ref{alg:O2NC-detail} with smoothing radius $\widehat{\delta}$. Since the above bound is expressed in $\mathcal{O}(\cdot)$ notation, replacing $\delta$ by $\widehat{\delta}$ changes only absolute constants. Define
\begin{align*}
    \widehat{\vg}_{k_{\rm out}}
    \triangleq
    \frac{1}{M}
    \sum_{t=(k_{\rm out}-1)M+1}^{k_{\rm out}M}
    \nabla f_{\widehat{\delta}}(\vy_t).
\end{align*}
Then Theorem \ref{thm:O2NC-generic} yields $\E[\norm{\widehat{\vg}_{k_{\rm out}}}] \le \epsilon$. Moreover, for every $t$ in the output block,
\begin{align*}
    \norm{\bar \vy_{k_{\rm out}}-\vy_t}
    \le
    \frac{1}{M}\sum_{s=(k_{\rm out}-1)M+1}^{k_{\rm out}M}
    \norm{\vy_s-\vy_t}
    \le
    \widehat{\delta}
    =
    \frac{\delta}{2}.
\end{align*}
Therefore, Proposition \ref{prop:f-delta-smooth}(e) implies $\nabla f_{\widehat{\delta}}(\vy_t)\in \partial_{\widehat{\delta}} f(\vy_t)\subseteq \partial_{\delta}f(\bar \vy_{k_{\rm out}})$ for every $t$ in the output block, and hence
\begin{align*}
    \widehat{\vg}_{k_{\rm out}}
    =
    \frac{1}{M}\sum_{t=(k_{\rm out}-1)M+1}^{k_{\rm out}M}\nabla f_{\widehat{\delta}}(\vy_t)
    \in
    \partial_{\delta}f(\bar \vy_{k_{\rm out}}).
\end{align*}
Thus,
\begin{align*}
    \E\big[
        {\rm dist}(\vzero,\partial_{\delta}f(\bar \vy_{k_{\rm out}}))
    \big]
    \le
    \E[\norm{\widehat{\vg}_{k_{\rm out}}}]
    \le
    \epsilon.
\end{align*}

\textbf{Option II.}
By Lemma \ref{lem:one-point-o2nc-moment},
\begin{align*}
    \E[\norm{\vg_t}^2]
    \le
    G_{\rm II}^2
    =
    \frac{6d^2\sigma^2}{\delta^2}
    +
    \aireplace{385d^2L^2}{42d^2L^2}.
\end{align*}
Applying Theorem \ref{thm:O2NC-generic} with $G=G_{\rm II}$ yields
\begin{align*}
    \E[\norm{\bar \vg_{k_{\rm out}}}]
    \le
    \frac{M(\gamma+\delta L)}{\delta T}
    +
    \frac{2G_{\rm II}}{\sqrt{M}}.
\end{align*}
Thus, by choosing $M$ and $T$ as in part (b) of the theorem, we have
\begin{align*}
    \E[\norm{\bar \vg_{k_{\rm out}}}]
    \le
    \epsilon.
\end{align*}
Since Option II uses one initialization query and one new query per iteration,
\begin{align*}
    N_{\rm SZO}
    =
    T+1.
\end{align*}
Using the choice of $T$ and the bound $M=\mathcal{O}(G_{\rm II}^2\epsilon^{-2})$, we obtain
\begin{align*}
    N_{\rm SZO}
    &=
    T+1 \\
    &=
    \mathcal{O}
    \left(
        \frac{(\gamma+\delta L)G_{\rm II}^2}{\delta\epsilon^3}
    \right) \\
    &=
    \mathcal{O}
    \left(
        \frac{\gamma+\delta L}{\delta\epsilon^3}
        \left(
            \frac{d^2\sigma^2}{\delta^2}
            +
            d^2L^2
        \right)
    \right) \\
    &=
    \mathcal{O}
    \left(
        d^2\gamma \sigma^2\delta^{-3}\epsilon^{-3}
        +
        d^2 L\sigma^2\delta^{-2}\epsilon^{-3}
        +
        d^2L^2\gamma\delta^{-1}\epsilon^{-3}
        +
        d^2L^3\epsilon^{-3}
    \right).
\end{align*}
Set $\widehat{\delta} \triangleq \delta/2$ and run Algorithm~\ref{alg:O2NC-detail} with Option II and smoothing radius $\widehat{\delta}$. As above, this changes only absolute constants in the displayed $\mathcal{O}(\cdot)$ bound. Define
\begin{align*}
    \widehat{\vg}_{k_{\rm out}}
    \triangleq
    \frac{1}{M}
    \sum_{t=(k_{\rm out}-1)M+1}^{k_{\rm out}M}
    \nabla f_{\widehat{\delta}}(\vy_t).
\end{align*}
Then Theorem \ref{thm:O2NC-generic} yields $\E[\norm{\widehat{\vg}_{k_{\rm out}}}] \le \epsilon$. Moreover, for every $t$ in the output block,
\begin{align*}
    \norm{\bar \vy_{k_{\rm out}}-\vy_t}
    \le
    \frac{1}{M}\sum_{s=(k_{\rm out}-1)M+1}^{k_{\rm out}M}
    \norm{\vy_s-\vy_t}
    \le
    \widehat{\delta}
    =
    \frac{\delta}{2}.
\end{align*}
Therefore, Proposition \ref{prop:f-delta-smooth}(e) implies $\nabla f_{\widehat{\delta}}(\vy_t)\in \partial_{\widehat{\delta}} f(\vy_t)\subseteq \partial_{\delta}f(\bar \vy_{k_{\rm out}})$ for every $t$ in the output block, and hence
\begin{align*}
    \widehat{\vg}_{k_{\rm out}}
    =
    \frac{1}{M}\sum_{t=(k_{\rm out}-1)M+1}^{k_{\rm out}M}\nabla f_{\widehat{\delta}}(\vy_t)
    \in
    \partial_{\delta}f(\bar \vy_{k_{\rm out}}).
\end{align*}
Thus,
\begin{align*}
    \E\big[
        {\rm dist}(\vzero,\partial_{\delta}f(\bar \vy_{k_{\rm out}}))
    \big]
    \le
    \E[\norm{\widehat{\vg}_{k_{\rm out}}}]
    \le
    \epsilon.
\end{align*}
\end{proof}

\section{Detailed Pseudo-code of ZO-O2NC with Option II}
\label{app:option-ii-alg}
For completeness, Algorithm~\ref{alg:O2NC-onepoint} gives a detailed presentation of ZO-O2NC with Option II. It uses one initialization query to construct the previous feedback value and then one fresh stochastic zeroth-order query per iteration.
\begin{algorithm}[htbp]
\caption{ZO-O2NC with Option II}\label{alg:O2NC-onepoint}
\renewcommand{\algorithmicrequire}{\textbf{Input:}}
\renewcommand{\algorithmicensure}{\textbf{Output:}}
\begin{algorithmic}[1]
\Require smoothing radius $\delta>0$, block length $M \in \sN$, number of blocks $K \in \sN$, stepsize $\eta>0$
\State $D=\delta/M$, $T=KM$, $\vx_0=\vy_0=\vzero$
\State Sample $\vu_0\sim{\rm Unif}(\sS^{d-1})$, draw $\xi_0\sim\Xi(\vy_0+\delta\vu_0)$, and query $h_0=F(\vy_0+\delta\vu_0;\xi_0)$
\For{$k=1,\ldots,K$}
\State $\vDelta_{(k-1)M+1}=\vzero$
\For{$m=1,\ldots,M$}
\State $t=(k-1)M+m$
\State Sample $s_t \sim {\rm Unif}[0,1]$
\State $\vy_t=\vx_{t-1}+s_t\vDelta_t$
\State $\vx_t=\vx_{t-1}+\vDelta_t$
\State Sample $\vu_t\sim{\rm Unif}(\sS^{d-1})$, draw $\xi_t\sim\Xi(\vy_t+\delta\vu_t)$, and query $h_t=F(\vy_t+\delta\vu_t;\xi_t)$
\State $\vg_t=\frac{d}{\delta}(h_t-h_{t-1})\vu_t$
\State $\vDelta_{t+1}=\Pi_{\sB(\vzero,D)}(\vDelta_t-\eta \vg_t)$
\EndFor
\EndFor
\State Sample $k_{\rm out} \sim {\rm Unif}\{1,\ldots,K\}$ and return $\bar \vy_{k_{\rm out}}=\frac{1}{M}\sum_{t=(k_{\rm out}-1)M+1}^{k_{\rm out}M}\vy_t$
\end{algorithmic}
\end{algorithm}

\section{The Proof of Section~\ref{sec:smooth-extension}}
\label{app:proof-sec4}

\subsection{The Proof of Proposition \ref{prop:smooth-reduction}}
\label{app:proof-prop-smooth-reduction}
\begin{proof}
Define
\begin{align*}
    \bar{\vg}
    \triangleq
    \frac{1}{M}\sum_{i=1}^{M}\nabla f_\delta(\vy_i).
\end{align*}
Under Assumption \ref{asm:grad-lip}, differentiation under the expectation gives
\begin{align*}
    \nabla f_\delta(\bar{\vy})
    =
    \E_{\vu \sim {\rm Unif}(\sB^d)}[\nabla f(\bar{\vy}+\delta\vu)].
\end{align*}
Hence, for any $\vx,\vz \in \sR^d$,
\begin{align*}
    \norm{\nabla f_\delta(\vx)-\nabla f_\delta(\vz)}
    &\le
    \E\big[\norm{\nabla f(\vx+\delta\vu)-\nabla f(\vz+\delta\vu)}\big] \\
    &\le
    L_g\norm{\vx-\vz},
\end{align*}
so $\nabla f_\delta$ is $L_g$-Lipschitz continuous. Therefore,
\begin{align*}
    \norm{\nabla f_\delta(\bar{\vy})-\bar \vg}
    &\le
    \frac{1}{M}\sum_{t=1}^M
    \norm{\nabla f_\delta(\bar{\vy})-\nabla f_\delta(\vy_t)} \\
    &\le
    \frac{L_g}{M}\sum_{t=1}^M \norm{\bar{\vy}-\vy_t}
    \le
    L_g\delta.
\end{align*}
Therefore,
\begin{align*}
    \norm{\nabla f(\bar{\vy})-\nabla f_\delta(\bar{\vy})}
    &=
    \left\|
        \E_{\vu \sim {\rm Unif}(\sB^d)}
        \big[
            \nabla f(\bar{\vy})-\nabla f(\bar{\vy}+\delta\vu)
        \big]
    \right\| \\
    &\le
    \E\big[\norm{\nabla f(\bar{\vy})-\nabla f(\bar{\vy}+\delta\vu)}\big]
    \le
    L_g\delta \E[\norm{\vu}]
    \le
    L_g\delta,
\end{align*}
and the stated bound in part (a) follows from the triangle inequality.

Under Assumption \ref{asm:Hess-lip}, differentiation under the expectation yields
\begin{align*}
    \nabla^2 f_\delta(\bar{\vy})
    =
    \E_{\vu \sim {\rm Unif}(\sB^d)}[\nabla^2 f(\bar{\vy}+\delta\vu)].
\end{align*}
Hence, for any $\vx,\vz \in \sR^d$,
\begin{align*}
    \norm{\nabla^2 f_\delta(\vx)-\nabla^2 f_\delta(\vz)}
    &\le
    \E\big[
        \norm{\nabla^2 f(\vx+\delta\vu)-\nabla^2 f(\vz+\delta\vu)}
    \big] \\
    &\le
    L_H\norm{\vx-\vz},
\end{align*}
so $\nabla^2 f_\delta$ is $L_H$-Lipschitz continuous. Applying the second-order Taylor expansion to $\nabla f_\delta$ at $\bar{\vy}$ gives
\begin{align*}
    \nabla f_\delta(\vy_i)
    =
    \nabla f_\delta(\bar{\vy})
    +
    \nabla^2 f_\delta(\bar{\vy})(\vy_i-\bar{\vy})
    +
    \vr_i,
    \qquad
    \norm{\vr_i} \le \frac{L_H}{2}\norm{\vy_i-\bar{\vy}}^2.
\end{align*}
Averaging over $i$ and using $\frac{1}{M}\sum_{i=1}^M(\vy_i-\bar{\vy})=\vzero$, we obtain
\begin{align*}
    \bar \vg - \nabla f_\delta(\bar{\vy})
    =
    \frac{1}{M}\sum_{i=1}^M \vr_i,
\end{align*}
and therefore
\begin{align*}
    \norm{\nabla f_\delta(\bar{\vy})-\bar \vg}
    \le
    \frac{1}{M}\sum_{i=1}^M \norm{\vr_i}
    \le
    \frac{L_H}{2M}\sum_{i=1}^M \norm{\vy_i-\bar{\vy}}^2
    \le
    \frac{L_H\delta^2}{2}.
\end{align*}
Next, for every $\vu \in \sB^d$, the second-order Taylor expansion of $\nabla f$ at $\bar{\vy}$ gives
\begin{align*}
    \nabla f(\bar{\vy}+\delta\vu)
    =
    \nabla f(\bar{\vy})
    +
    \delta \nabla^2 f(\bar{\vy})\vu
    +
    \vq(\vu),
    \qquad
    \norm{\vq(\vu)} \le \frac{L_H\delta^2}{2}\norm{\vu}^2.
\end{align*}
Taking expectations over $\vu \sim {\rm Unif}(\sB^d)$ and using $\E[\vu]=\vzero$ and $\E[\norm{\vu}^2]\le 1$, we obtain
\begin{align*}
    \norm{\nabla f_\delta(\bar{\vy})-\nabla f(\bar{\vy})}
    =
    \norm{\E[\vq(\vu)]}
    \le
    \E[\norm{\vq(\vu)}]
    \le
    \frac{L_H\delta^2}{2}.
\end{align*}
The stated bound in part (b) again follows from the triangle inequality.
\end{proof}

\subsection{The Proof of Theorem \ref{thm:smooth-gradient}}
\label{app:proof-thm-smooth-gradient}
\begin{proof}
Fix either Option I or Option II. Run Algorithm~\ref{alg:O2NC-detail} with target accuracy $\epsilon/2$ in Theorem~\ref{thm:O2NC-NSNC}. For the selected output block, define
\begin{align*}
    \bar \vg_{k_{\rm out}}
    \triangleq
    \frac{1}{M}
    \sum_{t=(k_{\rm out}-1)M+1}^{k_{\rm out}M}
    \nabla f_\delta(\vy_t).
\end{align*}
By Theorem \ref{thm:O2NC-NSNC}, the output $\bar \vy_{k_{\rm out}}$ satisfies
\begin{align*}
    \E[\norm{\bar \vg_{k_{\rm out}}}] \le \frac{\epsilon}{2}.
\end{align*}
Moreover, the proof of Theorem \ref{thm:O2NC-NSNC} gives
\begin{align*}
    \norm{\vy_t-\bar \vy_{k_{\rm out}}} \le \delta,
    \qquad
    t=(k_{\rm out}-1)M+1,\ldots,k_{\rm out}M.
\end{align*}
Applying Proposition \ref{prop:smooth-reduction}(a) yields
\begin{align*}
    \E[\norm{\nabla f(\bar \vy_{k_{\rm out}})}]
    \le
    \E[\norm{\bar \vg_{k_{\rm out}}}] + 2L_g\delta
    \le
    \frac{\epsilon}{2}+2L_g\delta.
\end{align*}
With $\delta=\epsilon/(4L_g)$, the right-hand side is at most $\epsilon$. For Option I, Theorem \ref{thm:O2NC-NSNC} gives
\begin{align*}
    N_{\rm SZO}
    =
    \mathcal{O}
    \left(
        \frac{\gamma+\delta L}{\delta(\epsilon/2)^3}
        \left(
            \frac{d^2\sigma^2}{\delta^2}
            +
            dL^2
        \right)
    \right)
\end{align*}
and for Option II,
\begin{align*}
    N_{\rm SZO}
    =
    \mathcal{O}
    \left(
        \frac{\gamma+\delta L}{\delta(\epsilon/2)^3}
        \left(
            \frac{d^2\sigma^2}{\delta^2}
            +
            d^2L^2
        \right)
    \right).
\end{align*}
Substituting $\delta=\epsilon/(4L_g)$ shows that both bounds reduce to
\begin{align*}
    N_{\rm SZO}
    =
    \gO\left(d^2\epsilon^{-6}\right),
\end{align*}
where constants depending on $\gamma,\sigma,L,$ and $L_g$ are absorbed into the $\gO(\cdot)$ notation.
\end{proof}

\subsection{The Proof of Theorem \ref{thm:smooth-hessian}}
\label{app:proof-thm-smooth-hessian}
\begin{proof}
Fix either Option I or Option II. Run Algorithm~\ref{alg:O2NC-detail} with target accuracy $\epsilon/2$ in Theorem~\ref{thm:O2NC-NSNC}. For the selected output block, define
\begin{align*}
    \bar \vg_{k_{\rm out}}
    \triangleq
    \frac{1}{M}
    \sum_{t=(k_{\rm out}-1)M+1}^{k_{\rm out}M}
    \nabla f_\delta(\vy_t).
\end{align*}
By Theorem \ref{thm:O2NC-NSNC},
\begin{align*}
    \E[\norm{\bar \vg_{k_{\rm out}}}] \le \frac{\epsilon}{2},
    \qquad
    \norm{\vy_t-\bar \vy_{k_{\rm out}}} \le \delta
\end{align*}
for all $t$ in the selected block. Therefore, Proposition \ref{prop:smooth-reduction}(b) gives
\begin{align*}
    \E[\norm{\nabla f(\bar \vy_{k_{\rm out}})}]
    \le
    \E[\norm{\bar \vg_{k_{\rm out}}}] + L_H\delta^2
    \le
    \frac{\epsilon}{2}+L_H\delta^2.
\end{align*}
With $\delta=\sqrt{\epsilon/(2L_H)}$, the right-hand side is at most $\epsilon$. For Option I, Theorem \ref{thm:O2NC-NSNC} gives
\begin{align*}
    N_{\rm SZO}
    =
    \mathcal{O}
    \left(
        \frac{\gamma+\delta L}{\delta(\epsilon/2)^3}
        \left(
            \frac{d^2\sigma^2}{\delta^2}
            +
            dL^2
        \right)
    \right)
\end{align*}
and for Option II,
\begin{align*}
    N_{\rm SZO}
    =
    \mathcal{O}
    \left(
        \frac{\gamma+\delta L}{\delta(\epsilon/2)^3}
        \left(
            \frac{d^2\sigma^2}{\delta^2}
            +
            d^2L^2
        \right)
    \right).
\end{align*}
Substituting $\delta=\sqrt{\epsilon/(2L_H)}$ shows that both bounds reduce to
\begin{align*}
    N_{\rm SZO}
    =
    \gO\left(d^2\epsilon^{-9/2}\right),
\end{align*}
where constants depending on $\gamma,\sigma,L,$ and $L_H$ are absorbed into the $\gO(\cdot)$ notation.
\end{proof}

\section{Finding Goldstein Stationary Point by Stochastic Gradient Descent}
\label{app:sgd-baseline}
\citet{hikima2025zeroth} use mini-batch zeroth-order estimators together with stochastic gradient descent.
For completeness, we record the corresponding guarantee for finding a $(\delta,\epsilon)$-Goldstein stationary point by applying stochastic gradient descent directly to the smoothed surrogate $f_\delta$.

Let
\begin{align*}
    V_\delta
    \triangleq
    \frac{d^2\sigma^2}{2\delta^2}
    +
    \aireplace{16\sqrt{2\pi}dL^2}{2dL^2},
    \qquad
    \beta_\delta
    \triangleq
    \frac{c\sqrt{d}L}{\delta},
\end{align*}
where $c>0$ is the universal constant from Proposition \ref{prop:f-delta-smooth}. Then $f_\delta$ is $\beta_\delta$-smooth.

For a given iterate $\vx_t$ and mini-batch size $B \in \sN$, define the mini-batch two-point estimator by
\begin{align}
    \label{eq:mb-two-point-estimator}
    \widehat{\vg}_t
    \triangleq
    \frac{1}{B}\sum_{j=1}^B
    \vg_{\text{2pt}}(\vx_t,\vu_{t,j},\xi_{t,j}^{+},\xi_{t,j}^{-}),
\end{align}
where, conditionally on $\mathcal{F}_t$, the tuples
\[
(\vu_{t,j},\xi_{t,j}^{+},\xi_{t,j}^{-}),
\qquad
j=1,\ldots,B,
\]
are sampled independently, with
\[
\vu_{t,j} \sim {\rm Unif}(\sS^{d-1}),
\qquad
\xi_{t,j}^{+} \sim \Xi(\vx_t+\delta\vu_{t,j}),
\qquad
\xi_{t,j}^{-} \sim \Xi(\vx_t-\delta\vu_{t,j}).
\]

\begin{algorithm}[htbp]
\caption{Mini-batch Two-Point SGD on $f_\delta$}\label{alg:mb-sgd}
\renewcommand{\algorithmicrequire}{\textbf{Input:}}
\renewcommand{\algorithmicensure}{\textbf{Output:}}
\begin{algorithmic}[1]
\Require smoothing radius $\delta>0$, mini-batch size $B \in \sN$, number of iterations $T \in \sN$, stepsize $\eta>0$
\State $\vx_0=\vzero$
\For{$t=0,\ldots,T-1$}
\State Construct $\widehat{\vg}_t$ by \eqref{eq:mb-two-point-estimator}
\State $\vx_{t+1}=\vx_t-\eta \widehat{\vg}_t$
\EndFor
\State Sample $t_{\rm out}\sim {\rm Unif}\{0,\ldots,T-1\}$ and return $\vx_{t_{\rm out}}$
\end{algorithmic}
\end{algorithm}

\begin{lem}\label{lem:mb-two-point}
Suppose Assumptions \ref{asm:estimator} and \ref{asm:lip_function} hold. Then, for every $t \ge 0$,
\begin{align*}
    \E[\widehat{\vg}_t \mid \mathcal{F}_t]
    &=
    \nabla f_\delta(\vx_t),~~~
    \E\big[
        \norm{\widehat{\vg}_t-\nabla f_\delta(\vx_t)}^2
        \mid
        \mathcal{F}_t
    \big]\le
    \frac{V_\delta}{B},
\end{align*}
and consequently
\begin{align*}
    \E\big[
        \norm{\widehat{\vg}_t}^2
        \mid
        \mathcal{F}_t
    \big]
    \le
    \norm{\nabla f_\delta(\vx_t)}^2
    +
    \frac{V_\delta}{B}.
\end{align*}
\end{lem}
\begin{proof}
For each $j=1,\ldots,B$, let
\begin{align*}
    \widehat{\vg}_{t,j}
    \triangleq
    \vg_{\text{2pt}}(\vx_t,\vu_{t,j},\xi_{t,j}^{+},\xi_{t,j}^{-}).
\end{align*}
Then
\begin{align*}
    \widehat{\vg}_t
    =
    \frac{1}{B}\sum_{j=1}^B \widehat{\vg}_{t,j}.
\end{align*}
By Lemma \ref{lem:zo-var}, for every $j$,
\begin{align*}
    \E[\widehat{\vg}_{t,j}\mid \mathcal{F}_t]
    =
    \nabla f_\delta(\vx_t),
    \qquad
    \E[\norm{\widehat{\vg}_{t,j}}^2\mid \mathcal{F}_t]
    \le
    V_\delta.
\end{align*}
Therefore,
\begin{align*}
    \E[\widehat{\vg}_t\mid \mathcal{F}_t]
    =
    \frac{1}{B}\sum_{j=1}^B \E[\widehat{\vg}_{t,j}\mid \mathcal{F}_t]
    =
    \nabla f_\delta(\vx_t).
\end{align*}
Moreover, conditional independence gives
\begin{align*}
    \E\big[
        \norm{\widehat{\vg}_t-\nabla f_\delta(\vx_t)}^2
        \mid
        \mathcal{F}_t
    \big]
    &=
    \frac{1}{B^2}
    \sum_{j=1}^B
    \E\big[
        \norm{\widehat{\vg}_{t,j}-\nabla f_\delta(\vx_t)}^2
        \mid
        \mathcal{F}_t
    \big] \\
    &\le
    \frac{1}{B^2}
    \sum_{j=1}^B
    \E\big[
        \norm{\widehat{\vg}_{t,j}}^2
        \mid
        \mathcal{F}_t
    \big] \\
    &\le
    \frac{V_\delta}{B}.
\end{align*}
Finally,
\begin{align*}
    \E\big[
        \norm{\widehat{\vg}_t}^2
        \mid
        \mathcal{F}_t
    \big]
    &=
    \norm{\E[\widehat{\vg}_t\mid \mathcal{F}_t]}^2
    +
    \E\big[
        \norm{\widehat{\vg}_t-\E[\widehat{\vg}_t\mid \mathcal{F}_t]}^2
        \mid
        \mathcal{F}_t
    \big] \\
    &\le
    \norm{\nabla f_\delta(\vx_t)}^2
    +
    \frac{V_\delta}{B}.
\end{align*}
\end{proof}

\begin{thm}\label{thm:mb-sgd}
Suppose Assumptions \ref{asm:estimator}, \ref{asm:lower}, and \ref{asm:lip_function} hold, 
then Algorithm \ref{alg:mb-sgd} can find a $(\delta,\epsilon)$-Goldstein stationary point of $f$ using
\begin{align*}
    N_{\rm SZO}
    =
    \mathcal{O}
    \left(
        (\gamma+\delta L)
        \left(
            d^{5/2}L\sigma^2\delta^{-3}
            +
            d^{3/2}L^3\delta^{-1}
        \right)
        \epsilon^{-4}
    \right)
\end{align*}
stochastic zeroth-order oracle calls.
\end{thm}
\begin{proof}
Since $f_\delta$ is $\beta_\delta$-smooth, for every $t$,
\begin{align*}
    f_\delta(\vx_{t+1})
    \le
    f_\delta(\vx_t)
    +
    \inner{\nabla f_\delta(\vx_t)}{\vx_{t+1}-\vx_t}
    +
    \frac{\beta_\delta}{2}\norm{\vx_{t+1}-\vx_t}^2.
\end{align*}
Substituting $\vx_{t+1}=\vx_t-\eta \widehat{\vg}_t$ gives
\begin{align*}
    f_\delta(\vx_{t+1})
    \le
    f_\delta(\vx_t)
    -
    \eta \inner{\nabla f_\delta(\vx_t)}{\widehat{\vg}_t}
    +
    \frac{\beta_\delta\eta^2}{2}
    \norm{\widehat{\vg}_t}^2.
\end{align*}
Taking conditional expectation with respect to $\mathcal{F}_t$ and using Lemma \ref{lem:mb-two-point}, we obtain
\begin{align*}
    \E[f_\delta(\vx_{t+1})\mid \mathcal{F}_t]
    &\le
    f_\delta(\vx_t)
    -
    \eta \norm{\nabla f_\delta(\vx_t)}^2 
    +
    \frac{\beta_\delta\eta^2}{2}
    \left(
        \norm{\nabla f_\delta(\vx_t)}^2
        +
        \frac{V_\delta}{B}
    \right) \\
    &=
    f_\delta(\vx_t)
    -
    \eta
    \left(
        1-\frac{\beta_\delta\eta}{2}
    \right)
    \norm{\nabla f_\delta(\vx_t)}^2
    +
    \frac{\beta_\delta\eta^2V_\delta}{2B}.
\end{align*}
Since $\eta \le 1/\beta_\delta$, we have $1-\beta_\delta\eta/2 \ge 1/2$, and hence
\begin{align*}
    \E[f_\delta(\vx_{t+1})\mid \mathcal{F}_t]
    \le
    f_\delta(\vx_t)
    -
    \frac{\eta}{2}
    \norm{\nabla f_\delta(\vx_t)}^2
    +
    \frac{\beta_\delta\eta^2V_\delta}{2B}.
\end{align*}
Taking total expectation and summing over $t=0,\ldots,T-1$ yield
\begin{align*}
    \frac{\eta}{2}
    \sum_{t=0}^{T-1}
    \E\big[
        \norm{\nabla f_\delta(\vx_t)}^2
    \big]
    \le
    f_\delta(\vx_0)-f_\delta^*
    +
    \frac{\beta_\delta\eta^2TV_\delta}{2B}.
\end{align*}
By Proposition \ref{prop:f-delta-smooth}(a) and Assumption \ref{asm:lower},
\begin{align*}
    f_\delta(\vx_0)-f_\delta^*
    \le
    f(\vx_0)-f^*+L\delta
    \le
    \gamma+\delta L.
\end{align*}
Therefore,
\begin{align*}
    \frac{1}{T}
    \sum_{t=0}^{T-1}
    \E\big[
        \norm{\nabla f_\delta(\vx_t)}^2
    \big]
    \le
    \frac{2(\gamma+\delta L)}{\eta T}
    +
    \frac{\beta_\delta\eta V_\delta}{B}.
\end{align*}
Since $t_{\rm out}$ is sampled uniformly from $\{0,\ldots,T-1\}$, this implies
\begin{align*}
    \E\big[
        \norm{\nabla f_\delta(\vx_{t_{\rm out}})}^2
    \big]
    \le
    \frac{2(\gamma+\delta L)}{\eta T}
    +
    \frac{\beta_\delta\eta V_\delta}{B}.
\end{align*}
Now choose
\begin{align*}
    \eta
    =
    \frac{1}{\beta_\delta},
    \qquad
    B
    =
    \left\lceil
        \frac{2V_\delta}{\epsilon^2}
    \right\rceil,
    \qquad
    T
    =
    \left\lceil
        \frac{4\beta_\delta(\gamma+\delta L)}{\epsilon^2}
    \right\rceil.
\end{align*}
Then
\begin{align*}
    \frac{2(\gamma+\delta L)}{\eta T}
    \le
    \frac{\epsilon^2}{2},
    \qquad
    \frac{\beta_\delta\eta V_\delta}{B}
    =
    \frac{V_\delta}{B}
    \le
    \frac{\epsilon^2}{2},
\end{align*}
and thus
\begin{align*}
    \E\big[
        \norm{\nabla f_\delta(\vx_{t_{\rm out}})}^2
    \big]
    \le
    \epsilon^2.
\end{align*}
By Proposition \ref{prop:f-delta-smooth}(e),
\begin{align*}
    \nabla f_\delta(\vx_{t_{\rm out}})
    \in
    \partial_\delta f(\vx_{t_{\rm out}}),
\end{align*}
and therefore
\begin{align*}
    {\rm dist}(\vzero,\partial_\delta f(\vx_{t_{\rm out}}))
    \le
    \norm{\nabla f_\delta(\vx_{t_{\rm out}})}.
\end{align*}
Taking expectations and applying Jensen's inequality give
\begin{align*}
    \E\big[
        {\rm dist}(\vzero,\partial_\delta f(\vx_{t_{\rm out}}))
    \big]
    \le
    \E\big[
        \norm{\nabla f_\delta(\vx_{t_{\rm out}})}
    \big]
    \le
    \left(
        \E\big[
            \norm{\nabla f_\delta(\vx_{t_{\rm out}})}^2
        \big]
    \right)^{1/2}
    \le
    \epsilon.
\end{align*}
Finally, since each iteration uses $2B$ stochastic zeroth-order oracle calls,
\begin{align*}
    N_{\rm SZO}
    =
    2BT
    =
    \mathcal{O}
    \left(
        \frac{(\gamma+\delta L)\beta_\delta V_\delta}{\epsilon^4}
    \right).
\end{align*}
Substituting
\[
\beta_\delta = \frac{c\sqrt{d}L}{\delta},
\qquad
V_\delta = \frac{d^2\sigma^2}{2\delta^2}+\aireplace{16\sqrt{2\pi}dL^2}{2dL^2}
\]
yields
\begin{align*}
    N_{\rm SZO}
    =
    \mathcal{O}
    \left(
        (\gamma+\delta L)
        \left(
            d^{5/2}L\sigma^2\delta^{-3}
            +
            d^{3/2}L^3\delta^{-1}
    \right)
    \epsilon^{-4}
    \right),
\end{align*}
as claimed.
\end{proof}